\documentclass[a4paper,12pt]{article}

\usepackage[left=2cm,right=2cm, top=2cm,bottom=2cm,bindingoffset=0cm]{geometry}

\usepackage{verbatim}
\usepackage{amsmath}
\usepackage{amsthm}
\usepackage{amssymb}
\usepackage{delarray}
\usepackage{cite}
\usepackage{hyperref}
\usepackage{mathrsfs}

\newcommand{\al}{\alpha}

\newcommand{\ga}{\gamma}
\newcommand{\de}{\delta}
\newcommand{\la}{\lambda}
\newcommand{\om}{\omega}
\newcommand{\ee}{\varepsilon}

\newcommand{\vv}{\varphi}
\newcommand{\iy}{\infty}

\theoremstyle{plain}

\newtheorem{thm}{Theorem}
\newtheorem{lem}{Lemma}

\theoremstyle{remark}

\DeclareMathOperator*{\Res}{Res}

 \allowdisplaybreaks

\begin{document}

\begin{center}
{\large\bf An inverse problem for the non-self-adjoint matrix Sturm-Liouville operator 
}
\\[0.2cm]
{\bf Natalia Bondarenko} \\[0.2cm]
\end{center}

\vspace{0.5cm}

{\bf Abstract.} The inverse problem of spectral analysis for the 
non-self-adjoint matrix Sturm-Liouville operator on a finite interval is investigated.
We study properties of the spectral characteristics for the considered operator,
and provide necessary and sufficient conditions for the solvability of the inverse problem.
Our approach is based on the constructive solution of the inverse problem by the
method of spectral mappings. The characterization of the spectral data in the self-adjoint case
is derived as a corollary of the main result.   
\medskip


\medskip


\vspace{1cm}

{\large \bf 1. Introduction and main results} \\

Consider the boundary value problem $L = L(Q(x), h, H)$ for the matrix Sturm-Liouville equation
\begin{equation} \label{eqv}
    \ell Y: = -Y''+ Q(x) Y = \la Y, \quad x \in (0, \pi),                                           
\end{equation}
with the boundary conditions
\begin{equation} \label{BC}
    U(Y) := Y'(0) - h Y(0) = 0, \quad V(Y) := Y'(\pi) + H Y(\pi) = 0.                               
\end{equation}

Here $Y(x) = [y_k(x)]_{k = \overline{1, m}}$ is a column vector,
$\la$ is the spectral parameter, and $Q(x) = [Q_{jk}(x)]_{j, k =
\overline{1, m}}$, where $Q_{jk}(x) \in L_2(0, \pi)$ are
complex-valued functions. We will subsequently refer to the matrix
$Q(x)$ as the \textit{potential}. The boundary conditions are
given by the matrices $h = [h_{jk}]_{j, k = \overline{1, m}}$, $H
= [H_{jk}]_{j, k = \overline{1, m}}$, where $h_{jk}$ and $H_{jk}$
are complex numbers. 

In this paper, we study the inverse problem of the spectral theory
for the matrix Sturm-Liouville operator, given by \eqref{eqv}, \eqref{BC}.
Inverse problems consist in recovering differentail operators from their spectral characteristics.
Such problems have many applications in science and engineering.

Inverse problems for the scalar Sturm-Liouville equation ($m = 1$) have been studied fairly completely
(see monographs \cite{Mar77, Lev84, PT87, FY01}). The matrix case is a natural generalization of
the scalar one. A significant contribution in the inverse problem theory for the matrix 
operators was made by Z.S.~Agranovich and V.A.~Marchenko \cite{AM60}, who studied the matrix Sturm-Liouville
operator on the half-line. For the inverse problem on the finite interval, a constructive solution was presented by V.A.~Yurko \cite{Yur06} in the case of the simple spectrum. 
Then D.~Chelkak and E.~Korotayev \cite{CK09}
have given the characterization of the spectral data (necessary and sufficient conditions) for the matrix Sturm-Liouville operator
with asymptotically simple spectrum, which is a strong restriction.
Necessary and sufficient conditions and an algorithm for the solution
in the general case, without any restrictions on the behavior of the spectrum,
provided in \cite{Bond11}. 
Ya.V. Mykytyuk and N.S. Trush \cite{MT10} obtained characterization of the spectral data for the potential 
from the Sobolev class $W_2^{-1}$.

All the previous works on the necessary and sufficient conditions for matrix Sturm-Liouville operators
deal with the self-adjoint case: when the matrices $Q$, $h$ and $H$ are Hermitian. In this paper,
we study the non-self-adjoint case. We develop the approach of \cite{Bond11}, based on the
method of spectral mappings \cite{FY01, Yur02}. This method allows to reduce an inverse problem 
to a so-called main equation, which is a linear equation in a suitable Banach space of infinite sequences. 
The reduction works for non-self-adjoint operators
just as well as for self-adjoint ones. 
Moreover, by necessity one can prove, that the main equation is uniquely solvable.
However, by sufficiency it is necessary to require its solvability even 
in the scalar case (see the example in \cite[Section 1.6.3]{FY01}).
For the non-self-adjoint scalar Sturm-Liouville operator,
a constructive solution of the inverse problem by the method of spectral mappings and
necessary and sufficient conditions 
were obtained by S.A. Buterin, C.-T. Shieh and V.A. Yurko
\cite{But07, BYS13}.
In this paper, we generalize their results, and get necessary and sufficient conditions for the spectral data 
of the matrix Sturm-Liouville operator.

Proceed to the formulation of the main results.
Let $\vv(x, \la)$ and $S(x, \la)$ be matrix-solutions of equation~\eqref{eqv} under the initial conditions
$$
    \vv(0, \la) = I_m, \quad \vv'(0, \la) = h, \quad S(0, \la) = 0_m, \quad S'(0, \la) = I_m.
$$
where $I_m$ is the identity $m \times m$ matrix, $0_m$ is the zero $m \times m$ matrix.
The function $\Delta(\la) := \det V(\vv)$ is called the
\textit{characteristic function} of the boundary value problem
$L$. The zeros of the entire function $\Delta(\la)$ coincide with
the eigenvalues of $L$.

Let $\om$ be some $m \times m$ matrix. We will write
$L(Q(x), h, H) \in A(\om)$, if the problem $L$ has a potential
from $L_2(0, \pi)$ and $h + H + \frac{1}{2} \int_0^{\pi} Q(x) \,
dx = \om$. In the self-adjoint case, the matrix 
$h + H + \frac{1}{2} \int_0^{\pi} Q(x) \,dx$ is diagonalizable 
by the unitary transform. 
In the general case, it is not true,
but we restrict ourselves to the class of diagonalizable matrices.
Then without loss of generality we can assume that 
\begin{equation*} 
    L \in A(\om), \quad \om \in \mathcal{D} = \{\om \colon \om = \mbox{diag}
\{\om_1, \ldots, \om_m\}\}.
\end{equation*}
One can achieve this condition applying the standard unitary transform.

Before we proceed to asymptotics, let us agree to denote by
$\{ \kappa_n \}$ different sequences from~$l_2$.

\begin{lem} \label{lem:asymptrho} 
Let $L \in A(\om)$, $\om \in \mathcal{D}$. The boundary
value problem $L$ has a countable set of eigenvalues $\{\la_{nq}
\}_{n \ge 0, q = \overline{1, m}}$, and
\begin{equation} \label{asymptrho}
     \rho_{nq} := \sqrt{\la_{nq}} = n + \frac{\om_q}{\pi n} + \frac{\kappa_n}{n},
     \quad q = \overline{1, m}.                            
\end{equation}
Here the eigenvalues are counted with their multiplicities, which they have as zeros of the entire characteristic 
function $\Delta(\la)$.
\end{lem}

Since the matrix $\om$ is diagonal, the proof of Lemma~\ref{lem:asymptrho} repeats the proof
of \cite[Lemma~1]{Bond11}.

Let $\Phi(x, \la) = [\Phi_{jk}(x, \la)]_{j, k = \overline{1, m}}$
be a matrix-solution of equation~\eqref{eqv} under the boundary conditions
$U(\Phi) = I_m$, $V(\Phi) = 0_m$. We call $\Phi(x, \la)$ the
\textit{Weyl solution} for $L$. Put $M(\la) := \Phi(0, \la)$. The
matrix $M(\la) = [M_{jk}(\la)]_{j,k = \overline{1, m}}$ is called
the \textit{Weyl matrix} for $L$. The notion of the Weyl matrix
is a generalization of the notion of the Weyl function ($m$-function)
for the scalar case (see \cite{Mar77}, \cite{FY01}).
The Weyl functions and their generalizations often appear in applications and in pure
mathematical problems, and they are natural spectral characteristics in the inverse problem
theory for various classes of differential operators.

Using the definition for $\Phi(x, \la)$ and $M(\la)$, one can easily check that
%\begin{equation} \label{expPhi}
%\Phi(x, \la) = S(x, \la) + \vv(x, \la) M(\la),
%\end{equation}
\begin{equation} \label{formM}
    M(\la) = -(V(\vv))^{-1} V(S).                                                          
\end{equation}
The matrix-function $M(\la)$ is meromorphic in $\la$ with poles at the eigenvalues
$\{\la_{nq}\}$ of $L$. In general, the poles can be multiple, but we put the following restriction. 

\medskip

{\bf Assumption 1.} All the poles of the matrix-function $M(\la)$ are simple.

\medskip

Note that Assumption~1 corresponds to the case, when the operator does not have associated functions
(see \cite{Naimark}). If there is a finite number of multiple poles, one
can use the approach of \cite{But07, BYS13}.

Define the {\it weight matrices}:
$$
    \alpha_{nq} := \mathop{\mathrm{Res}}_{\la = \la_{nq}}M(\la).
$$

\medskip

{\bf Assumption 2.} The sequence of the matrices $\{ \al_{nq} \}$
is bounded in a matrix norm: $\| \al_{nq}\| \le C$, for all $n \ge 0$,  $q = \overline{1, m}$.

\medskip

For definiteness, here and below we consider the following matrix norm
\begin{equation} \label{matrixnorm}
\| A \| = \max_{1 \le j \le m} \sum_{k = 1}^m |a_{jk} |, \quad A = [a_{jk}]_{j, k = \overline{1, m}}.
\end{equation}

We say that the boundary value problem $L$ belongs to the class $A_{1,2}(\om)$, if
$L \in A(\om)$ and $L$ satisfies Assumptions~1 and 2.

Let $\{ \la_{n_k q_k} \}_{k \geq 0}$ be all the distinct
eigenvalues from the collection $\{ \la_{n q}\}_{n \geq 0, q =
\overline{1, m}}$. Put
$$
     \alpha'_{n_k q_k} := \alpha_{n_k q_k}, \, k \geq 0, \quad \alpha'_{nq} = 0_m,
     \, (n, q) \notin \{ (n_k, q_k )\}_{k \geq 0}.
$$

Fix the numbers $1 = m_1 < m_2 < \dots < m_p$ so that
$\{ \omega_{m_s} \}_{s = 1}^p$ are all the distinct values 
in the collection $\{ \omega_q \}_{q = 1}^m $.
Let $J_s = \{ q \colon \om_q = \om_{m_s} \}$, and
$\alpha_n^{(s)} = \sum\limits_{q \in J_s} \alpha'_{nq}$, $s = \overline{1, p}$.
\footnote{In the case of multiple eigenvalues, the same weight matrices $\al_{nq}$ occur in $\Lambda$ multiple times. 
To count each residue in the sum only once, we use the notation $\al'_{nq}$.
}
Analogously to $\kappa_n$, denote by $\{ K_n \}$ different
matrix sequences, such that norms of these matrices form sequences from $l_2$.

\begin{lem} \label{lem:asymptalpha1}
Let $L \in A_{1,2}(\om)$, $\om \in \mathcal{D}$. Then the
following relations hold
\begin{equation} \label{asymptalpha1}
    \alpha_n^{(s)} = \frac{2}{\pi} I^{(s)} + K_n,
    \quad s = \overline{1, p}, \quad n \ge 0,                     
\end{equation}
\begin{equation} \label{asymptalpha2}
(I_m - I^{(s)}) \al_{nq} = K_n, \quad n \ge 0, \: s = \overline{1, p}, \: q \in J_s,
\end{equation}
where
$$
    I^{(s)} = [I^{(s)}_{jk}]_{j, k = \overline{1, m}}, \quad I^{(s)}_{jk} = \left\{
                                                            \begin{array}{cc}
                                                               1, & j = k \in J_s, \\
                                                               0, & otherwise.
                                                            \end{array} \right.
$$
\end{lem}

Put $\al_n := \sum\limits_{s = 1}^p \al_n^{(s)} = \sum\limits_{q = 1}^m \al'_{nq}$.

\begin{lem} \label{lem:asymptalpha2}
Let $L \in A_{1,2}(\om)$, $\om \in \mathcal{D}$. Then the following relation holds
\begin{equation} \label{asymptalpha3}
    \al_n = \frac{2}{\pi} I_m + \frac{K_n}{n}, \quad n \ge 0.
\end{equation}
\end{lem}

The data $\Lambda := \{\la_{nq}, \alpha_{nq} \}_{n \geq 0, \,q =
\overline{1, m}}$ are called the \textit{spectral data} of the
problem $L$.
Consider the following inverse problem.

\medskip

{\bf Inverse Problem 1.} Given the spectral data $\Lambda$, construct $Q$, $h$ and $H$.

\medskip

Describe the general strategy of our method.
Suppose we know the spectral data $\Lambda$ of some unknown boundary value problem 
$L \in A_{1,2}(\om)$, $\om \in \mathcal{D}$.
Choose an arbitrary model boundary value problem $\tilde L = L(\tilde Q(x), \tilde h, \tilde H) \in A_{1,2}(\om)$ 
(for example, one can take $\tilde Q(x) = \frac{2}{\pi} \om$, $\tilde h = 0_m$,
$\tilde H = 0_m$). We agree that if a certain symbol $\gamma$
denotes an object related to $L$, then the corresponding symbol
$\tilde \gamma$ with tilde denotes the analogous object related to
$\tilde L$. 

Denote $\la_{nq0} = \la_{nq}, \quad \la_{nq1} = \tilde \la_{nq}$, $n \ge 0$, $q = \overline{1, m}$. 
Let $\psi(x) = [\vv(x, \la_{nqi})]_{n \ge 0, q = \overline{1, m}, i = 0, 1}$,
$\tilde \psi(x) = [\tilde \vv(x, \la_{nqi})]_{n \ge 0, q = \overline{1, m}, i = 0, 1}$.
It is shown in Section~4, that for each fixed $x \in [0, \pi]$, $\psi(x)$ satisfies {\it the main equation}
\begin{equation} \label{main0}
\tilde \psi(x) = \psi(x) (I + \tilde R(x))    
\end{equation}
in a suitable Banach space $B$ of infinite bounded matrix sequences.
Here $I$ is the identity operator in $B$, and the operator $\tilde R(x)$ is constructed
by the model problem $\tilde L$ and two sets of spectral data $\Lambda$, $\tilde \Lambda$.
Solving the main equation, one can recover the potential $Q$ and the coefficients of
the boundary conditions $h$ and $H$ by Algorithm~1, provided in Section~4.
Using the main equation, we obtain necessary and sufficient conditions for
spectral data of the problem $L$ from $A_{1,2}(\om)$.

We will write $\{\la_{nq}, \alpha_{nq} \}_{n \geq 0, q =
\overline{1, m}} \in \mbox{Sp}$, if $\la_{nq}$ are complex numbers,
$\al_{nq}$ are $m \times m$ matrices, and for $\la_{nq} = \la_{kl}$ we
always have $\alpha_{nq} = \alpha_{kl}$.

\begin{thm} \label{thm:1}
Let $\om \in \mathcal{D}$. For data $\{\la_{nq}, \alpha_{nq}\}_{n \geq 0, q = \overline{1, m}} \in \mbox{Sp}$ to be
the spectral data for a certain problem $L \in A_{1,2}(\om)$ it is
necessary and sufficient to satisfy the following conditions.

(A) The asymptotics \eqref{asymptrho}, \eqref{asymptalpha1}, 
\eqref{asymptalpha2} \eqref{asymptalpha3} are valid, and Assumption~2 holds for $\{ \al_{nq} \}$.

(R) The ranks of the
matrices $\al_{nq}$ coincide with the multiplicities of the corresponding
values $\la_{nq}$.

(M) The main equation \eqref{main0} is uniquely solvable.

Condition (M) holds for any choice of a model problem $\tilde L \in A_{1,2}(\om)$ by necessity and 
for at least one problem $\tilde L$ by sufficiency.

\end{thm}

Of particular interest are those cases, when the solvability of the main equation can be proved
or easily checked, namely, the self-adjoint case, the case of finite pertrubations of the spectral data
and the case of small pertrubations \cite{But07, BYS13}. As a corollary of Theorem~\ref{thm:1},
be derive a result for the self-adjoint case: $Q = Q^{\dagger}$, $h = h^{\dagger}$, $H = H^{\dagger}$
(the symbol $\dagger$ stands for the conjugate transpose). 
Finite pertrubations and small pertrubations can also be studied analogously to the scalar case.
Note that in the self-adjoint case, the problem $L$ always belongs to the class $A_{1,2}(\om)$ with a diagonalizable $\om$
(see Section~7). Condition (M) can be proved with help of the simplier condition (E),
so we obtain the following result.

\begin{thm} \label{thm:2}
Let $\om = \om^{\dagger} \in \mathcal{D}$.
For data $\{\la_{nq}, \alpha_{nq}\}_{n \geq 0, q = \overline{1, m}} \in \mbox{Sp}$ to be
the spectral data for a certain self-adjoint problem $L \in A(\om)$ it is
necessary and sufficient to satisfy the following conditions. 

(A) The asymptotics \eqref{asymptrho}, \eqref{asymptalpha1}, 
\eqref{asymptalpha2} \eqref{asymptalpha3} are valid.

(R) The ranks of the
matrices $\al_{nq}$ coincide with the multiplicities of the corresponding
values $\la_{nq}$.

(S) All $\la_{nq}$ are real, $\alpha_{nq} = (\alpha_{nq})^\dagger$,
$\alpha_{nq} \geq 0$ for all $n \geq 0$, $q = \overline{1, m}$.
 
(E) For any row vector $\gamma(\la)$ that is entire in $\la$, and
that satisfy the estimate
$$
    \gamma(\la) = O(\exp(|\mbox{Im}\, \sqrt{\la}|\pi)), \quad |\la| \to \iy,
$$
if $\gamma(\la_{nq})\alpha_{nq} = 0$ for all $n \geq 0$, $q =
\overline{1, m}$, then $\gamma(\la) \equiv 0$.\footnote{
The letters, denoting the conditions, have the following meanings: (A) {\bf A}symptotics, 
(R) {\bf R}anks, (M) solvability of the {\bf M}ain equation, (S) {\bf S}elf-adjointness, 
(E) ``{\bf E}ntire function condition''. The conditions (C) {\bf C}ompleteness, and
(PW) {\bf P}aley-{\bf W}iener class condition, appear later (in Section 5).
}
\end{thm}

Note that Assumption~2 in (A) is not necessary in the self-adjoint case, because
it follows from \eqref{asymptalpha3} and the condition $\al_{nq} \ge 0$.

As we have already mentioned, the characterization of the spectral data of the self-adjoint
matrix Sturm-Liouville operator was obtained earlier in \cite{Bond11}.
But the work \cite{Bond11} contains a technical mistake in asymptotics of the weight matrices.
In this paper, using the method of \cite{Bond11}, we obtain correct necessary and sufficient conditions for the self-adjoint case 
(Theorem~\ref{thm:2}) as a corollary of 
a more general result (Theorem~\ref{thm:1}).

The paper is organized as follows. 
At first we study algebraic and analytical properties of the spectral characteristics
(\textit{Section~2}) and prove Lemmas~\ref{lem:asymptalpha1} and
\ref{lem:asymptalpha2} with asymptotic formulas for the weight matrices (\textit{Section~3}).
In \textit{Section 4}, we derive the main equation in a suitable Banach space
and provide a constructive algorithm for the solution of Inverse Problem~1.     
We also prove the unique solvability of the main equation, and this finish the proof
of the necessity in Theorem~1.
Further, in \textit{Section 5}, we discuss the connection between the conditions
(M), (E), (C) and (PW). Namely, they are connected as follows: $(M) \Rightarrow (E) \Rightarrow (PW) \Leftrightarrow (C)$.
\textit{Section 6} devoted to the sufficiency in Theorem~\ref{thm:1}.
In \textit{Section 7}, we collect the results concerning the self-adjoint case, and prove
Theorem~\ref{thm:2}. We also give a reformulation of Theorem~\ref{thm:2}, using
the completeness of some system of vector functions (C).

\textbf{Notation.}
Along with $L$ we consider the boundary value problem $L^* = L^*(Q(x), h, H)$ in the form
\begin{equation} \label{defL*}
\begin{array}{c}
\ell^* Z := -Z'' + Z Q(x) = \la Z, \quad x \in (0, \pi), \\
U^*(Z) := Z'(0) - Z(0)h = 0, \quad V^*(Z) := Z'(\pi) + Z(\pi) H = 0,
\end{array}
\end{equation}
where $Z$ is a row vector. Let $\tilde L^* = L^*(\tilde Q(x), \tilde h, \tilde H)$.
We agree that if a symbol $\ga$ denotes an object related to $L$, then $\ga^*$ and $\tilde \ga^*$
denote corresponding objects related to $L^*$ and $\tilde L^*$, respectively.

We consider the space of complex column $m$-vectors $\mathbb{C}^m$ with the norm
$$
    \| Y \| = \max_{1 \le j \le m} |y_j|, \quad Y = [y_j]_{j = \overline{1, m}},
$$
the space of complex $m \times m$ matrices $\mathbb{C}^{m \times m}$ with the corresponding induced norm
\eqref{matrixnorm}, and the space of row vectors $\mathbb{C}^{m, T}$.
We use the spaces $L_2((0, \pi), \mathbb{C}^m)$, 
$L_2((0, \pi), \mathbb{C}^{m, T})$ and $L_2((0, \pi), \mathbb{C}^{m \times m})$
of column vectors, row vectors and matrices, respectively, with entries from $L_2(0, \pi)$.
The Hilbert spaces $L_2((0, \pi), \mathbb{C}^m)$ and $L_2((0, \pi), \mathbb{C}^{m, T})$
are equipped with the following scalar products
$$
    (Y, Z) = \int_0^{\pi} Y^{\dagger}(x) Z(x) \, dx = \int_0^{\pi} \sum_{j = 1}^m \bar{y}_j(x) z_j(x) \, dx,
$$
$$
    (Y, Z) = \int_0^{\pi} Y(x) Z^{\dagger}(x) \, dx,
$$
respectively.
Denote $\langle Y, Z \rangle = Y' Z - Y Z'$.

Put $\rho := \sqrt \la$, $\mbox{Re} \, \rho \geq 0$, $\tau :=
\mbox{Im}\, \rho$.
In estimates and asymptotics, we use the same symbol $C$ for different constants
independent of $x$, $\rho$, etc.

\bigskip

{\large \bf 2. Properties of the spectral data}\\

The results of this section are valid for any boundary value problem $L$, satisfying
Assumption~1. First, we prove an alternative formulation of this assumption.

\begin{lem} \label{lem:Vphi}
Assumption~1 is equivalent to the condition, that all the poles of
the matrix function $(V(\vv(x, \la)))^{-1}$ in the $\la$-plane are simple.
\end{lem} 

\begin{proof}
Suppose that $\la_0$ is a nonsimple pole of $(V(\vv))^{-1}$, namely
$$
    (V(\vv))^{-1} = \frac{A_{-k}}{(\la - \la_0)^k} + \dots + \frac{A_{-1}}{\la - \la_0} + A_0 + \dots, \quad k > 1, \quad A_k \ne 0_m,
$$ 
in a neirborhood of $\la_0$. The matrix-function $V(S(x, \la))$ is analytical:
$V(S(x, \la)) = V(S(x, \la_0)) + \frac{d}{d\la} V(S(x, \la_0)) (\la - \la_0) + \dots$.
If $\la_0$ is a simple pole of the Weyl matrix $M(\la)$, then
$A_{-k} V(S(x, \la_0)) = 0_m$. The matrix-function $I = (V(\vv))^{-1} V(\vv)$ is entire,
therefore we also have $A_{-k} V(\vv(x, \la_0)) = 0_m$. 
Since the columns of the matrices $\vv(x, \la)$ and $S(x, \la)$ form a fundamental system of solutions
of equation \eqref{eqv}, every solution $\psi(x, \la)$ of this equation can be represented as their linear combination:
$\psi(x, \la) = \vv(x, \la) A + S(x, \la) B$, and it also satisfies the relation $A_{-k} V(\psi(x, \la_0)) = 0$.
But if we choose the solution $\psi(x, \la_0)$, satisfying the initial conditions $\psi(\pi, \la_0) = 0_m$,
$\psi'(\pi, \la_0) = A_{-k}^{\dagger}$, we get $A_{-k} V(\psi(x, \la_0)) \ne 0$. The contradiction shows,
that the simplicity of the poles of $(V(\vv))^{-1}$ follows from the simplicity of the poles of $M(\la)$.
The inverse is obvious. 
\end{proof}

\begin{lem}
The zeros of the characteristic function $\Delta(\la)$ coincide with the eigenvalues of the
boundary value problem $L$. 
The multiplicity of each zero $\la_0$ of the function
$\Delta(\la)$ equals to the multiplicity of the corresponding eigenvalue
(by the multiplicity of the eigenvalue we mean the number of the
corresponding linearly independent vector eigenfunctions).
\end{lem}

\medskip

\begin{proof} 1. Let $\la_0$ be an eigenvalue of $L$, and let $Y^0$
be an eigenfunction corresponding to $\la_0$. Let us show that
$Y^0(x) = \vv(x, \la_0) Y^0(0)$. Clearly, $Y^0(0) = \vv(0, \la)
Y^0(0)$. It follows from $U(Y^0) = 0$ that ${Y^0}'(0) = h Y^0(0) =
\vv(0, \la) Y^0(0)$. Thus, $Y^0(x)$ and $\vv(x, \la_0)Y^0(0)$ are
the solutions for the same initial value problem for equation \eqref{eqv}.
Consequently, they are equal.

2. Let us have exactly $k$ linearly independent eigenfunctions
$Y^1$, $Y^2$, \ldots, $Y^k$ corresponding to the eigenvalue
$\la_0$. Choose the invertible $m \times m$ matrix $C$ such that
the first $k$ columns of $\vv(x, \la_0) C$ coincide with the
eigenfunctions. Consider $Y(x, \la) := \vv(x, \la) C$, $Y(x, \la)
= [Y_q(x, \la)]_{q = \overline{1, m}}$, $Y_q(x, \la_0) = Y^q(x),
\, q = \overline{1, k}$. Clearly, the zeros of $\Delta_1(\la)
:= \det V(Y) = \det V(\vv) \cdot \det C$ coincide with the zeros
of $\Delta(\la)$ counting with their multiplicities. Note that
$\la = \la_0$ is a zero of each of the columns $V(Y_1)$, \dots,
$V(Y_k)$. Hence, if $\la_0$ is the zero of the determinants
$\Delta_1(\la)$ and $\Delta(\la)$ with the multiplicity $p$, then $p
\geq k$.

3. Suppose that $p > k$. Rewrite $\Delta_1(\la)$ in the form
$$
    \Delta_1(\la) = (\la - \la_0)^k \Delta_2(\la),
$$
$$
    \Delta_2(\la) = \det \left[\frac{V(Y_1)}{\la-\la_0},
    \ldots, \frac{V(Y_k)}{\la-\la_0}, V(Y_{k + 1}), \ldots, V(Y_m)\right].
$$
In view of our supposition, we have $\Delta_2(\la_0) = 0$, i.\,e. there exist
not all zero coefficients $\alpha_q$, $q = \overline{1, m}$
such that
\begin{equation} \label{smeq1}
   \sum_{q = 1}^k \alpha_q \frac{d V(Y_q(x, \la_0))}{d\la} +
   \sum_{q = k + 1}^m \alpha_q V(Y_q(x, \la_0)) = 0.                                         
\end{equation}

If $\alpha_q = 0$ for $q = \overline{1, k}$, then the function
$$
    Y^+(x, \la) := \sum_{q = k + 1}^{m} \alpha_q Y_q(x, \la)
$$
for $\la = \la_0$ is an eigenfunction, corresponding to $\la_0$,
that is linearly independent with $Y^q$, $q = \overline{1, k}$.
Since the eigenvalue $\la_0$ has exactly $k$ corresponding
eigenfunctions, we arrive at a contradiction.

Otherwise we consider the function
$$
    Y^+(x, \la) := \sum_{q = 1}^{k} \alpha_q Y_q(x, \la) +
    (\la - \la_0) \sum_{q = k + 1}^{m} \alpha_q Y_q(x, \la).
$$
Now we plan to use the simplicity of the poles of $(V(\vv))^{-1}$,
following from Assumption~1 by Lemma~\ref{lem:Vphi}.
Recall the following well-known fact (see \cite[Lemma 2.2.1]{AM60}):

\smallskip

{\it
The inverse $(V(\vv))^{-1}$ has a simple pole at $\la = \la_0$ if and only if the relations at $\la = \la_0$:
\begin{equation} \label{factAM}
\begin{array}{c}
 V(\vv)a = 0, \\
 \frac{d}{d \la} V(\vv)a + V(\vv)b = 0,
\end{array}
\end{equation}
where $a$ and $b$ are constant vectors, yield $a = 0$.
}

\smallskip

The function $Y^+$ has the form $Y^+(x, \la) = V(\vv) a + (\la - \la_0) V(\vv) b$, $a \ne 0$.
In view of \eqref{smeq1}, the relations \eqref{factAM} are satisfied, and we arrive at a contradiction
with Assumption~1. Thus, $\Delta_2(\la_0) = 0$ and $p = k$. 

\end{proof}

\begin{lem} \label{lem:ranks}
The ranks of the residue-matrices of the Weyl matrix $M(\la)$ 
coincide with the multiplicities of the corresponding eigenvalues of $L$.
\end{lem}

Under Assumption~1, the proof of Lemma~\ref{lem:ranks} does not differ from the proof in the self-adjoint case
(see \cite[Lemma 4]{Bond11}.

Now let us consider the problem $L^*$, defined by \eqref{defL*}. 
It is easy to check that
\begin{equation} \label{BC*}
\langle Z, Y \rangle_{x = 0} = U^*(Z) Y(0) - Z(0) U(Y), \quad
\langle Z, Y \rangle_{x = \pi} = V^*(Z) Y(\pi) - Z(\pi) V(Y).
\end{equation}
where $\langle Z, Y \rangle = Z'Y - ZY'$.
If $Y(x, \la)$ and $Z(x, \la)$ satisfy the equations
$\ell Y(x, \la) = \la Y(x, \la)$, $\ell Z(x, \mu) = \mu Z(x, \mu)$, respectively, then
\begin{equation} \label{wron}
\frac{d}{dx} \langle Z(x, \mu), Y(x, \la) \rangle = (\la - \mu) Z(x, \mu) Y(x, \la),
\end{equation}

Introduce the matrices 
$\vv^*(x, \la)$, $S^*(x, \la)$ and $\Phi^*(x, \la)$, satisfying the equation 
$\ell^* Z = \la Z$ and the conditions
$$
    \vv^*(0, \la) = {S^*}'(0, \la) = U^*(\Phi^*) = I_m, \quad
    {\vv^*}'(0, \la) = h, \quad S^*(0, \la) = V^*(\Phi^*) = 0_m.
$$
Denote $M^*(\la) := \Phi^*(0, \la)$.

In view of \eqref{wron}, the expression $\langle \Phi^*(x, \la), \Phi(x, \la) \rangle$ does not depend on $x$.
Using \eqref{BC*}, we obtain
$$
    \langle \Phi^*(x, \la), \Phi(x, \la) \rangle_{x = 0} = M(\la) - M^*(\la), \quad
    \langle \Phi^*(x, \la), \Phi(x, \la) \rangle_{x = \pi} = 0_m.   
$$
Hence 
\begin{equation} \label{eqM}
    M(\la) \equiv M^*(\la).
\end{equation}
and consequently, the spectral data of the problems $L$ and $L^*$ coincide.

\begin{lem} \label{lem:sym} 
Let $\la_0$, $\la_1$ be eigenvalues of $L$,
$\la_0 \neq \la_1$, and $\alpha_i =
\mathop{\mathrm{Res}}\limits_{\la = \la_{i}}M(\la)$, $i = 0, 1$.
The following relations hold
$$
    \alpha_0 \int\limits_0^{\pi} \vv^*(x, \la_0) \vv(x, \la_0) \, dx \, \alpha_0 = \alpha_0,
$$
$$
    \alpha_0 \int\limits_0^{\pi} \vv^*(x, \la_0) \vv(x, \la_1) \, dx \, \alpha_1 = 0_m.
$$
\end{lem}

\begin{proof}
Using \eqref{BC*} and \eqref{wron}, we derive
$$
    \int_0^{\pi} \vv^*(x, \la_0) \vv(x, \la_0) \, dx =
    \lim_{\la \to \la_0} \frac{\langle \vv^*(x, \la_0), \vv(x, \la) \rangle|_0^{\pi}}{\la - \la_0}
$$
$$
    = \lim_{\la \to \la_0} \frac{V^*(\vv^*(x, \la_0)) \vv(\pi, \la) -
    \vv^*(\pi, \la_0) V(\vv(x, \la))}{\la - \la_0}.
$$
In view of \eqref{formM}, the product $V(\vv) M(\la)$ is an entire function of $\la$. 
Taking its residues at $\la_0$,
we get
\begin{equation} \label{prodVal0}
V(\vv(x, \la_0)) \al_0 = 0_m.
\end{equation} 
Similarly $\al_0 V(\vv^*(x, \la_0)) = 0_m$.
Consequently, we calculate
$$
   \alpha_0 \int\limits_0^{\pi} \vv^*(x, \la_0) \vv(x, \la_0) \, dx \, \alpha_0 =
   \alpha_0 \vv^*(\pi, \la_0) \lim_{\la \to \la_0} \frac{V(\vv(x, \la))}{\la - \la_0}
$$
$$
    \times \lim_{\la \to \la_0} (\la - \la_0) (V(\vv(x, \la)))^{-1} V(S(x, \la))
   = \alpha_0 \vv^*(\pi, \la_0) V(S(x, \la_0))
$$
$$
   = -\alpha_0 \langle \vv^*(x, \la_0), S(x, \la_0) \rangle_{x = \pi}
   = -\alpha_0 \langle \vv^*(x, \la_0), S(x, \la_0) \rangle_{x = 0} = \alpha_0.
$$
Similarly one can derive the second relation of the lemma.
\end{proof}

\begin{lem} \label{lem:E}
Let $\{ \la_{nq}, \al_{nq}\}_{n \ge 0, q = \overline{1, m}}$ be the spectral data of the problem $L$,
satisfying Assumptions~1 and~2. Then (E) is valid.
\end{lem}

\begin{proof}
Let $\gamma(\la)$ be a function described in (E). 
In view of \eqref{prodVal0}, we have 
\begin{equation} \label{prodVal}
    V(\vv(x, \la_{nq})) \al_{nq} = 0_m,  \quad n \ge 0, \: q = \overline{1, m}. 
\end{equation}
Since
$$\mbox{rank}\; V(\vv(x, \la_{nq})) + \mbox{rank}\; \alpha_{nq} = m$$
 and $\gamma(\la_{nq})\alpha_{nq} = 0$, we get
$\gamma(\la_{nq}) = C_{nq} V(\vv(x, \la_{nq}))$, i.\,e. the row
$\gamma(\la_{nq})$ is a linear combination of the rows of the
matrix $V(\vv(x, \la_{nq}))$ (here $C_{nq}$ is a row of
coefficients). Consider
$$f(\la) = \gamma(\la) (V(\vv(x, \la)))^{-1}.$$
The matrix-function $(V(\vv(x, \la)))^{-1}$ has simple poles in
$\la = \la_{nq}$, therefore, we calculate
$$
   \mathop{\mathrm{Res}}_{\la = \la_{nq}} f(\la) = \gamma(\la_{nq}) \mathop{\mathrm{Res}}_{\la = \la_{nq}}(V(\vv(x, \la)))^{-1}
$$
$$
   = C_{nq} \lim_{\la \to \la_{nq}} V(\vv(x, \la)) \lim_{\la \to \la_{nq}}(\la - \la_{nq})
  (V(\vv(x, \la)))^{-1} = 0.
$$
Hence, $f(\la)$ is entire. It is easy to show that
\[
    \| (V(\vv(x, \la)))^{-1}\| \leq C_{\de} |\rho|^{-1} \exp(-|\tau| \pi), \quad \rho \in G_{\de},
\]
where $G_{\de} = \{\rho \colon |\rho - k| \geq \de, k = 0, 1,
2, \dots \}$, $\de > 0$. From this we conclude that $\|
f(\la)\| \leq \frac{C}{|\rho|}$ in $G_{\de}$. By the maximum
principle this estimate is valid in the whole $\la$-plane. Using
Liouville`s theorem, we obtain $f(\la) \equiv 0$. Consequently,
$\gamma(\la) \equiv 0$.
\end{proof}

\medskip

{\large \bf 3. Asymptotics} \\

In this section, we prove Lemmas~\ref{lem:asymptalpha1} and \ref{lem:asymptalpha2},
providing asymptotic formulas for the weight matrices $\al_{nq}$.

\begin{lem}
For $|\rho| \to \infty$, the following asymptotic formulae hold
\begin{equation} \label{asymptVphi}
     V(\vv) = - \rho \sin \rho \pi \cdot I_m + \om \cos \rho \pi + \kappa(\rho),
\end{equation}
$$
     \kappa(\rho) = \frac{1}{2} \int\limits_{0}^{\pi}Q(t) \cos \rho(\pi - 2 t)\,dt +
     O\left(\frac{\exp(|\tau|\pi)}{\rho}\right);
$$
\begin{equation} \label{asymptVS}
   V(S) = \cos \rho \pi \cdot I_m + \frac{\sin \rho \pi}{\rho} \om_0 + \frac{\kappa_0(\rho)}{\rho},
   \quad \om_0 = H + \frac{1}{2} \int_0^{\pi} Q(t)\,dt,
\end{equation}
$$
    \kappa_0(\rho) = -\frac{1}{2} \int_0^{\pi} \sin \rho(\pi - 2 t) Q(t) \, dt +
    O\left(\frac{\exp(|\tau|\pi)}{\rho}\right).
$$
\end{lem}

\begin{proof}
The assertion of the lemma immediately follows from the standard asymptotics:
\begin{equation} \label{asymptphi}
\vv(x, \la) = \cos \rho x  \cdot I_m + \frac{\sin \rho x}{\rho} Q_1(x) + \frac{1}{2 \rho} \int_0^x \sin \rho(x - 2 t)Q(t) \, dt
+ O \left( \frac{\exp(|\tau|x)}{\rho^2} \right),
\end{equation}
\begin{equation} \label{asymptdphi}
  \vv'(x, \la) = -\rho \sin \rho x \cdot I_m + \cos \rho x Q_1(x) + \frac{1}{2} \int_0^x \cos \rho(x - 2 t)Q(t)\, dt +
+ O \left( \frac{\exp(|\tau|x)}{\rho} \right),
\end{equation}
where $Q_1(x) = h + \frac{1}{2}\int_0^x Q(t)\, dt$, and
$$
    S(x, \la) = \frac{\sin \rho x}{\rho} \cdot I_m - \frac{\cos \rho x}{2 \rho^2}\int_0^x Q(t) \, dt + 
    \frac{1}{2 \rho^2} \int_0^x \cos \rho(x - 2t) Q(t)\,dt + O\left(\frac{\exp(|\tau|x)}{\rho^3}\right),
$$
$$
    S'(x, \la) = \cos \rho x \cdot I_m + \frac{\sin \rho x}{2 \rho}\int_0^x Q(t)\,dt - 
    \frac{1}{2\rho} \int_0^x \sin \rho(x - 2 t) Q(t) \, dt + O\left( \frac{\exp(|\tau|x)}{\rho^2} \right).
$$
\end{proof}

\begin{proof}[Proof of Lemma~\ref{lem:asymptalpha1}]
Consider the contour 
$$
    \ga_n^{(s)} := \left\{ \la \colon \la = n^2 + \frac{2}{\pi} \mu, \: |\mu - \omega_{m_s}| = R := \frac{1}{2}\min_{j,k}|\om_j - \om_k| \right\}.
$$
Then for sufficiently large $n$, by virtue of \eqref{asymptrho} and the residue theorem,
\begin{equation} \label{smeq5}
    \al_n^{(s)} = \frac{1}{2 \pi i}\int\limits_{\ga_n^{(s)}} M(\la) \, d \la.
\end{equation}

Further in this proof, we fix $s = \overline{1, p}$ and a sufficiently large $n$, 
and consider only $\la = n^2 + \dfrac{2}{\pi} \mu \in \ga_n^{(s)}$. 
Taking a square root, we get 
\begin{equation} \label{rhoform}
    \rho = \sqrt \la =  n + \frac{\mu}{\pi n} + \frac{\kappa_n(\mu)}{n}.
\end{equation}
Here and below $\{ \kappa_n(\mu) \}$ denotes different sequences, depending on $\mu$, but majorized by 
a constant sequence from $l_2$, independent of $\mu$: 
$$
    \forall \mu \colon \, |\mu - \om_{m_s}| = R, \quad |\kappa_n(\mu)| \le \kappa_n, \quad \sum_{n} \kappa_n^2 < \infty.
$$
Similarly, $\{ K_n(\mu) \}$ denotes sequences of matrices, whose norms form scalar sequences $\{ \kappa_n(\mu) \}$.

It follows from \eqref{asymptVphi} and \eqref{asymptVS}, that for $\la \in \ga_n^{(s)}$,
$$
    V(\vv) = (-1)^n (- \mu I_m + \om + K_n(\mu)), \quad V(S) = (-1)^n \left(I_m + \frac{K_n(\mu)}{n} \right).   
$$
Substitute this into \eqref{formM}:
$$
    M(\la) = (\mu I_m - \om + K_n(\mu))^{-1} \left(I_m + \frac{K_n(\mu)}{n} \right).
$$
Since $|\mu - \om_q| \ge R$ for all $q = \overline{1, m}$, the inverse $M_0(\mu) := (\mu I_m - \om)^{-1}$ is bounded, and
$$
    M(\la) - M_0(\mu) = K_n(\mu), \quad \frac{1}{2 \pi i} \int\limits_{\ga_n^{(s)}} (M(\la) - M_0(\mu)) \, d \la = K_n.
$$
We calculate
$$
    \frac{1}{2 \pi i} \int\limits_{\ga_n^{(s)}} M_0(\mu) \, d\la = 
    \frac{2}{\pi} \cdot \frac{1}{2 \pi i} \int\limits_{|\mu - \om_{m_s}| = R} (\mu I_m - \om)^{-1} \, d \mu
    = \frac{2}{\pi} I^{(s)}.    
$$
Together with \eqref{smeq5}, this gives \eqref{asymptalpha1}.

By \eqref{asymptrho} and \eqref{asymptVphi}, $V(\vv(x, \la_{nq})) = (-1)^n (\om - \om_q I_m + K_n)$. 
By Assumption~2, $\| \al_{nq} \| \le C$. Using these facts together with \eqref{prodVal}, we obtain
$(\om - \om_q I_m) \al_{nq} = K_n$. This relation yields \eqref{asymptalpha2}.

\end{proof}

\begin{lem} \label{lem:int}
Let a matrix $A$ be such that $\| A \| < R$. Then
$$
    \frac{1}{2\pi i} \int\limits_{|\mu| = R} (\mu I_m - A)^{-1} d \mu = I_m.
$$
\end{lem}

\begin{proof}
The matrix-function $F(\mu) = (\mu I_m - A)^{-1}$ is analytic outside the circle $|\mu| < R$.
Therefore,
$$
    \frac{1}{2 \pi i} \int_{|\mu| = R} F(\mu)\, d \mu = - \Res_{\mu = \infty} F(\mu).
$$
The Laurent series 
$$
    F(\mu) = \frac{1}{\mu} \left(I_m + \frac{A}{\mu} + \frac{A^2}{\mu^2} + \dots \right)
$$
converge uniformly when $|\mu| \ge R > \| A \|$. Therefore
$$
    \Res_{\mu = \infty} F(\mu) = -I_m,
$$
that yields the assertion of the lemma.
\end{proof}

\begin{proof}[Proof of Lemma~\ref{lem:asymptalpha2}]
Note that in fact, the asymptotics \eqref{asymptalpha3} with the remainder $K_n$ are already proved.
In order to improve this estimate, we will work with the remainder $\kappa(\rho)$ in \eqref{asymptVphi}.

Substituting the representation
\begin{multline*}
\vv'(x, \la) = -\rho \sin \rho x \cdot I_m + \cos \rho x Q_1(x) +  
\frac{1}{2} \int_0^x \cos \rho (x - 2t) Q(t) \, dt \\ +
\frac{\sin \rho x}{2\rho} \int_0^x Q(t)Q_1(t)\, dt - \frac{1}{2\rho}\int_0^x \sin \rho (x - 2t)Q(t)Q_1(t)\, dt \\ +
\frac{1}{2\rho} \int_0^x \cos \rho (x - t) Q(t) \int_0^t \sin \rho(t - 2s) Q(s) \, ds \, dt + O\left( \frac{\exp(|\tau|x)}{\rho^2} \right),
\end{multline*}
$$
   Q_1(x) := h + \frac{1}{2}\int_0^x Q(t)\, dt,
$$
and \eqref{asymptphi} into $V(\vv) = \vv'(\pi, \la) + H \vv(\pi, \la)$, we arrive at \eqref{asymptVphi} with
\begin{multline} \label{reprkappa}
\kappa(\rho) = \frac{1}{2} \int_0^{\pi} \cos \rho (\pi - 2 t) Q(t) \, dt + \frac{\sin \rho \pi}{\rho}
\left( \frac{1}{2} \int_0^{\pi} Q(t) Q_1(t) \, dt  + H Q_1(\pi) \right) \\ + \frac{1}{2 \rho} \int_0^{\pi} \sin \rho(\pi - 2 t)
\left(H Q(t) - Q(t) Q_1(t) \right)\, dt \\  + \frac{1}{2 \rho} \int_0^{\pi} \cos \rho (\pi - t) Q(t) 
\int_0^t \sin \rho (t - 2 s) Q(s) \, ds \, dt +  O\left( \frac{\exp(|\tau|\pi)}{\rho^2}\right).
\end{multline}

Consider the contour
$$
   \ga_n := \left\{ \la \colon \la = n^2 + \frac{2}{\pi} \mu, \: |\mu| = 3 \| \om \| \right\}.
$$
Further in this proof, we fix a sufficiently large $n$, 
such that
\begin{equation} 
    \al_n = \frac{1}{2 \pi i}\int\limits_{\ga_n} M(\la) \, d \la.
\end{equation}
and consider only $\la = n^2 + \dfrac{2}{\pi} \mu \in \ga_n$. 
Then the square root of $\la$ takes the form \eqref{rhoform}.

Substitute \eqref{rhoform} into \eqref{reprkappa}. Then the first integral in \eqref{reprkappa} equals
$$
    \frac{1}{2} \int_0^{\pi} \cos n (\pi - 2 t) Q(t) \, dt + \frac{K_n(\mu)}{n},
$$
and all the other terms are $\dfrac{K_n(\mu)}{n}$.
Then by \eqref{formM}, \eqref{asymptVphi}, \eqref{asymptVS}, we get
$$
    M(\la) = \left(\mu I_m - \om + L_n + \frac{K_n(\mu)}{n} \right)^{-1} \left( I_m + \frac{K_n(\mu)}{n}\right), \quad \la \in \ga_n,
$$
where $L_n$ is a matrix sequence independent of $\mu$, and $\{ \| L_n \|\} \in l_2$. 
Thus, for large $n$, $\| L_n \| + \| K_n(\mu) / n \| \le \| \om \|$ and the inverses
$\left(\mu I_m - \om + L_n + \frac{K_n(\mu)}{n} \right)^{-1}$ and 
$\left(\mu I_m - \om + L_n\right)^{-1}$ are bounded for $|\mu| = 3 \| \om \|$. Therefore
$$
    \frac{1}{2\pi i}\int\limits_{|\mu| = 3 \| \om \|} M(\la) \, d\mu = \frac{1}{2\pi i}
    \int\limits_{|\mu| = 3 \| \om\|} \left(\mu I_m - \om + L_n\right)^{-1} \, d\mu + \frac{K_n}{n}. 
$$
Applying Lemma~\ref{lem:int} to the right-hand side and changing $d \mu$ to $d \la$, we arrive at \eqref{asymptalpha3}.
\end{proof}

\medskip

{\large \bf 4. Solution of Inverse Problem 1} \\

Let the spectral data $\Lambda$ of the boundary value problem $L
\in A_{1, 2}(\om)$, $\om \in \mathcal{D}$, be given.

Denote
\begin{equation} \label{defD}
     D(x, \la, \mu) = \frac{ \langle \vv^*(x, \mu),
     \vv(x, \la)\rangle}{\la - \mu} = \int\limits_0^x
     \vv^*(t, \mu) \vv(x, \la) \, dt.                                       
\end{equation}

We choose an arbitrary model boundary value problem $\tilde L =
L(\tilde Q(x), \tilde h, \tilde H) \in A_{1,2}(\om)$ (for example, one
can take $\tilde Q(x) = \frac{2}{\pi} \om$, $\tilde h = 0_m$,
$\tilde H = 0_m$). Note that for this choice of the model problem $\om = \tilde \om$, therefore 
the eigenvalues $\tilde \la_{nq}$ and the weight matrices $\tilde \al_{nq}$ 
of $\tilde L$ satisfy the same asymptotic formulae
\eqref{asymptrho}, \eqref{asymptalpha1}, \eqref{asymptalpha2} and \eqref{asymptalpha3}, as $\la_{nq}$
and $\al_{nq}$. 
Put
\begin{equation} \label{defxi}
   \xi_n = \sum_{q = 1}^m |\rho_{nq} - \tilde \rho_{nq}| + \sum_{s = 1}^p \sum_{q \in J_s} |\rho_{nq} - \rho_{n m_s}| +
    \sum_{s = 1}^p \sum_{q \in J_s} |\tilde \rho_{nq} - \tilde \rho_{n m_s}| + 
    \sum_{s = 1}^p \frac{1}{n}\| \alpha_n^{(s)} - \tilde \alpha_n^{(s)}\| + \| \al_n - \tilde \al_n \|,
\end{equation}
then
\begin{equation} \label{defOmega}
    \Omega := \left( \sum_{n = 0}^{\iy} ((n + 1) \xi_n)^2\right)^{1/2} < \iy, \quad \sum_{n = 0}^{\iy} \xi_n < \iy.
\end{equation}

Denote
$$
    \begin{array}{c}
    \la_{nq0} = \la_{nq}, \quad \la_{nq1} = \tilde \la_{nq}, \quad \rho_{nq0} = \rho_{nq}, \quad \rho_{nq1} = \tilde \rho_{nq}, \quad
    \alpha'_{nq0} = \alpha'_{nq}, \quad \alpha'_{nq1} = \tilde \alpha'_{nq}, \\ 
    \vv_{nqi}(x) = \vv(x, \la_{nqi}), \quad
    \tilde \vv_{nqi}(x) = \tilde \vv(x, \la_{nqi}), \quad
    \vv^*_{nqi}(x) = \vv^*(x, \la_{nqi}), \quad
    \tilde \vv^*_{nqi}(x) = \tilde \vv^*(x, \la_{nqi}),     
    \\
    n \geq 0, \quad q = \overline{1, m}, \quad i = 0, 1.
    \end{array}
$$

By the standard way (see \cite[Lemma~1.6.2]{FY01}), using Schwarz's lemma, we get

\begin{lem} \label{lem:Schwarz}
The following estimates are valid for $x \in [0, \pi]$, $n, k \geq 0$, $q, l, r = \overline{1, m}$, $i, j, s = 0, 1$:
\begin{gather*}
   \|\vv_{nqi}(x)\| \leq C, \quad
   \|\vv_{n q i}(x) - \vv_{n l j}(x) \| \leq C |\rho_{nqi} - \rho_{nlj}|, \\
   \| D(x, \la_{nqi}, \la_{klj}) \| \le \frac{C}{|n - k| + 1}, \quad
   \| D(x, \la_{nqi}, \la_{klj}) - D(x, \la_{nqi}, \la_{krs}) \| \le \frac{C |\rho_{klj} - \rho_{krs}|}{|n - k| + 1}.
\end{gather*}
The analogous estimates are also valid for $\tilde \vv_{nqi}(x)$,
$\tilde D(x, \la_{nqi}, \la_{klj})$, as well as for similar matrix functions,
related to the problems $L^*$, $\tilde L^*$.
\end{lem}

The lemma similar to the following one has been proved in \cite{Yur06} by the contour integral method.

\begin{lem} \label{lem:contour}
The following relations hold
\begin{equation} \label{mainphila}
    \tilde \vv(x, \la) = \vv(x, \la) + \sum_{k = 0}^{\iy}
    \sum_{l = 1}^m \left(\vv_{kl0}(x) \al'_{kl0} \tilde D(x, \la, \la_{kl0}) -
    \vv_{kl1}(x) \al'_{kl1} \tilde D(x, \la, \la_{kl1})\right)                                
\end{equation}
$$
    \tilde D(x, \la, \mu) - D(x, \la,\mu) = \sum_{k = 0}^{\iy}
    \sum_{l = 1}^m \left(D(x, \la_{kl0}, \mu) \al'_{kl0} \tilde D(x, \la, \la_{kl0}) -
    D(x, \la_{kl1}, \mu) \al'_{kl1} \tilde D(x, \la, \la_{kl1})\right).
$$
Both series converge absolutely and uniformly with respect to $x \in [0, \pi]$ and
$\la$, $\mu$ on compact sets.
\end{lem}

Analogously one can obtain the following relation 
\begin{equation} \label{mainPhi}
    \tilde \Phi(x, \la) = \Phi(x, \la) +
    \sum_{k = 0}^{\iy} \sum_{l = 1}^m \sum_{j = 0}^1
    (-1)^j \vv_{klj}(x) \al'_{klj} \frac{\langle
    \tilde \vv^*_{klj}(x), \tilde \Phi(x, \la)
         \rangle}{\la - \la_{klj}}.                                                           
\end{equation}

It follows from Lemma~\ref{lem:contour} that
\begin{equation} \label{mainphi}
    \tilde \vv_{nqi}(x) = \vv_{nqi}(x) + \sum_{k = 0}^{\iy}
    \sum_{l = 1}^m (\vv_{kl0}(x) \al'_{kl0} \tilde D(x, \la_{nqi}, \la_{kl0}) -
    \vv_{kl1}(x) \al'_{kl1} \tilde D(x, \la_{nqi}, \la_{kl1})),                                                    
\end{equation}
\begin{multline} \label{mainsol}
    \al'_{\eta r j} \tilde D(x, \la_{nqi}, \la_{\eta r j}) -  \al'_{\eta r j} D(x, \la_{nqi}, \la_{\eta r j}) =
    \sum_{k = 0}^{\iy} \sum_{l = 1}^m \bigl(\al'_{\eta r j} D(x, \la_{kl0}, \la_{\eta r j}) \al'_{kl0} 
    \tilde D(x, \la_{nqi}, \la_{kl0}) \\ -
    \al'_{\eta r j} D(x, \la_{kl1}, \la_{\eta r j}) \al'_{kl1} \tilde D(x, \la_{nqi}, \la_{kl1}) \bigr).                                    
\end{multline}
for $n, \eta \geq 0$, $q, r = \overline{1, m}$, $i, j  = 0, 1$. 

For each fixed $x \in [0, \pi]$, the relation \eqref{mainphi} can be
considered as a system of linear equations with respect to
$\vv_{nqi}(x)$, $n \geq 0$, $q = \overline{1, m}$, $i = 0, 1$. But
the series in \eqref{mainphi} converges only ``with brackets''. Therefore, it
is not convenient to use \eqref{mainphi} as a main equation of the inverse
problem. Below we will transfer \eqref{mainphi} to a linear equation in a
corresponding Banach space of sequences.

Introduce collections $G_n = \{ \rho_{nqi}\}_{q = \overline{1, m}, i = 0, 1}$, $n \ge 0$.
Fix $n$ and, for convenience, renumerate the elements of the collection: $G_n = \{ g_i \}_{i = 1}^{2m}$.
Consider a finite-dimensional space 
$B(G_n) = \bigl\{f \colon G_n \to \mathbb{C}^{m \times m} \bigr\}$ of matrix-functions $f$,
such that $f(g_i) = f(g_j)$ if $g_i = g_j$, with the norm
$$
    \| f \|_{B(G_n)} = \max \left\{ \max_i \| f(g_i) \|, \max_{i, j\colon g_i \ne g_j} \| f(g_i) - f(g_j) \| \cdot |g_i - g_j|^{-1} \right\}.
$$
Introduce a Banach space of infinite row vectors
$$
    B = \{ f = \{ f_n \}_{n = 0}^{\infty} \colon f_n \in B(G_n), \: 
    \| f\|_{B} := \sup_{n \ge 0} \| f_n \|_{B(G_n)} < \infty  \}.
$$

Fix $x \in [0, \pi]$. Lemma~\ref{lem:Schwarz} gives the following estimates:
$$
    \| \vv(x, g_i^2) \| \le C, \quad \| \vv(x, g_i^2) - \vv(x, g_j^2) \| \le C | g_i - g_j |, \quad
    g_i, g_j \in G_n,   
$$ 
where the constant $C$ does not depend on $n$. Therefore, $\vv(x, \rho^2)$ forms an element of $B$:
$$
    \vv(x, \rho^2)_{| B} := \{ \vv(x, \rho^2)_{| G_n}\}_{n \ge 0} \in B, \quad
    \vv(x, \rho^2)_{| G_n} = \{ \vv(x, \la_{nqi}) \}_{q = \overline{1, m}, i = 0, 1}.
$$
Denote $\psi(x) := \vv(x, \rho^2)_{| B}$, $\tilde \psi(x) := \tilde \vv(x, \rho^2)_{| B}$.
Then \eqref{mainphi} and \eqref{mainsol} can be transformed into the following relations in the 
Banach space $B$:
\begin{equation} \label{main}
    \tilde \psi(x) = \psi(x) (I + \tilde R(x)),
\end{equation}
\begin{equation} \label{smeqR}
   \tilde R(x) - R(x) = R(x) \tilde R(x),
\end{equation}
where $I$ is the identity operator in $B$, and $R(x)$, $\tilde R(x)$ are linear operators, 
acting from $B$ to $B$.~\footnote{
The action of operators $R(x)$ and $\tilde R(x)$ is, in fact, a multiplication
of an infinite row vector to an infinite matrix. 
It is more convenient to write operators to the right of operands, 
to keep the correct order in elementwise multiplication, 
that is the noncommutative multiplication of $m \times m$ matrices. 
}
The explicit form of $\tilde R(x)$ and $R(x)$ can be derived from \eqref{mainphi} and \eqref{mainsol}.
Further we investigate the operator $R(x)$, the same properties for $\tilde R(x)$ can be obtained symmetrically.

According to \eqref{mainphi} and \eqref{mainsol}, the operator $R(x)$ acts on an arbitrary element 
$\psi = \{ \psi_k \}_{k = 0}^{\infty} \in B$
in the following way:
\begin{equation} \label{defRsimp}
    (\psi R(x))_n = \sum_{k = 0}^{\infty} \psi_k R_{k, n}(x), \quad R_{k, n} \colon B(G_k) \to B(G_n), \quad k, n \ge 0,
\end{equation}
\begin{equation} \label{defRkn}
    (\psi_k R_{k, n}(x))(\rho_{nqi}) = \sum_{l = 1}^{m} \left(\psi_k(\rho_{kl0}) \al'_{kl0} D(x, \la_{nqi}, \la_{kl0}) -
    \psi_k(\rho_{kl1}) \al'_{kl1} D(x, \la_{nqi}, \la_{kl1}) \right).   
\end{equation}

\begin{lem} \label{lem:Rbound}
The series in \eqref{defRsimp} converge in $B(G_n)$-norm and the operator $R(x)$ is bounded and, moreover, compact on $B$.
\end{lem}

\begin{proof}
Let $\psi = \{ \psi_k\}_{k = 0}^{\infty} \in B$. Fix $x \in [0, \pi]$ and $n, k \ge 0$. 
Denote $\psi_{klj} := \psi_k(\rho_{klj})$, 
$\eta_{nqi, k} := (\psi_k R_{k, n}(x))(\rho_{nqi})$. Let us show that 
\begin{equation} \label{smeqxi}
\| \eta_{nqi, k}\| \le \frac{ C \xi_k \| \psi_k \|_{B(G_k)}}{|n - k| + 1},  \quad q = \overline{1, m}, \: i = 0, 1,
\end{equation}
where $\xi_k$ was defined in \eqref{defxi} and the constant $C$ does not depend on $n$ and $k$.

Using \eqref{defRkn}, we derive
\begin{multline*}
    \eta_{nqi, k} = \sum_{l = 1}^m \Bigl[ (\psi_{kl0} - \psi_{kl1})\al'_{kl0} D(x, \la_{nqi}, \la_{kl0}) \\ +
    \psi_{kl1} \al'_{kl0} (D(x, \la_{nqi}, \la_{kl0}) -  D(x, \la_{nqi}, \la_{kl1})) + 
    \psi_{kl1} (\al'_{kl0} - \al'_{kl1}) D(x, \la_{nqi}, \la_{kl1})  \Bigr].
\end{multline*}
Since
$$
    \| \psi_{kl1} \| \le \| \psi_k \|_{B(G_k)}, \quad \| \psi_{kl0} - \psi_{kl1} \| \le |\rho_{kl0} - \rho_{kl1}| \| \psi_k \|_{B(G_k)} \le \xi_k \| \psi_k \|_{B(G_k)},
    \quad l = \overline{1, m},
$$
$\| \al'_{kl0} \| \le C$, and $D(x, \la_{nqi}, \la_{klj})$ satisfy estimates of Lemma~\ref{lem:Schwarz}, one
easily obtain the estimate \eqref{smeqxi} for the first two terms.

Recall that $J_s = \{q \colon \om_q = \om_{m_s} \}$, $s = \overline{1, p}$, are indices in
groups with equal terms $\om_q$ in asymptotics \eqref{asymptrho}.
Continue to work with the third term:
\begin{multline*}
\sum_{l = 1}^m \psi_{kl1} (\al'_{kl0} - \al'_{kl1}) D(x, \la_{nqi}, \la_{kl1})  =
\sum_{s = 1}^p \biggl[ \sum_{l \in J_s} (\psi_{kl1} - \psi_{k m_s 1}) (\al'_{kl0} - \al'_{kl1}) D(x, \la_{nqi}, \la_{kl1}) \\ +
\sum_{l \in J_s} \psi_{k m_s 1} (\al'_{kl0} - \al'_{kl1}) (D(x, \la_{nqi}, \la_{kl1}) - D(x, \la_{nqi}, \la_{k m_s 1})) + 
\psi_{k m_s 1} (\al_k^{(s)} - \tilde \al_k^{(s)}) D(x, \la_{nqi}, \la_{k m_s 1}) \biggr]
\end{multline*}
Applying the estimate
$$
    \| \psi_{kl1} - \psi_{k m_s 1} \| \le |\rho_{kl1} - \rho_{k m_s 1}| \| \psi_k \|_{B(G_k)} 
    \le \xi_k \| \psi_k \|_{B(G_k)}, \quad l \in J_s,
$$
estimates for $\al'_{klj}$ and Lemma~\ref{lem:Schwarz}, we arrive at \eqref{smeqxi} for the first two terms again
and continue to investigate the third one.
\begin{multline*}
\sum_{s = 1}^p \psi_{k m_s 1} (\al_k^{(s)} - \tilde \al_k^{(s)}) D(x, \la_{nqi}, \la_{k m_s 1}) =
\sum_{s = 1}^p (\psi_{k m_s 1} - \psi_{k 1 1}) (\al_k^{(s)} - \tilde \al_k^{(s)}) D(x, \la_{nqi}, \la_{k m_s 1}) \\ +
\sum_{s = 1}^p \psi_{k11} (\al_k^{(s)} - \tilde \al_k^{(s)})(D(x, \la_{nqi}, \la_{k m_s 1}) - D(x, \la_{nqi}, \la_{k11})) +
\psi_{k11} (\al_k - \tilde \al_k) D(x, \la_{nqi}, \la_{k11}).
\end{multline*}
Now we use the estimates
$$
    \| \psi_{k m_s 1} - \psi_{k11} \| \le |\rho_{k m_s 1} - \rho_{k11}| \| \psi_k \|_{B(G_k)} \le \frac{\| \psi_k \|_{B(G_k)}}{k},
$$
$$
    \| \al_k^{(s)} - \tilde \al_k^{(s)} \| \le k \xi_k, \quad \| \al_k - \tilde \al_k \| \le \xi_k,
$$
(following from \eqref{asymptalpha1}, \eqref{asymptalpha3} and similar asymptotics
for $\tilde \al_{nq}$) and Lemma~\ref{lem:Schwarz}. Finally we arrive at~\eqref{smeqxi}.

Analogously one can obtain the estimate
$$
\| \eta_{nqi, k} - \eta_{nlj, k}\| \le \frac{ C \xi_k \| \psi_k \|_{B(G_k)}|\rho_{nqi} - \rho_{nlj}|}{|n - k| + 1},  \quad q, l = \overline{1, m}, \: i, j = 0, 1.
$$
Together with \eqref{smeqxi}, this gives 
\begin{equation} \label{estR}
    \| R_{k, n}(x)\|_{B(G_k) \to B(G_n)} \le \frac{C \xi_k}{|n - k| + 1}, \quad k, n \ge 0,
\end{equation}
where the constant $C$ does not depend on $n$ and $k$.
Substitute \eqref{estR} into \eqref{defRsimp} and use \eqref{defOmega}:
$$
    \| \psi R(x) \|_B = \sup_{n \ge 0} \left\| (\psi R(x))_n\right\| \le \| \psi \|_B 
    \left( \sum_{k = 0}^{\infty} \frac{C \xi_k}{|n - k| + 1}\right) \le C \| \psi \|_B.
$$
Hence $\| R(x) \|_{B \to B} < \infty$.

The operator $R(x)$ can be approximated by a sequence of finite-dimensional operators in $B$.
Indeed, let $R^s_{k, n}(x) = R_{k, n}(x)$ for all $n \ge 0$, $0 \le k \le s$, and all the other components
of $R^s(x)$ equal zero. It is easy to show using \eqref{estR}, that $\lim\limits_{s \to \infty} \| R^s(x) - R(x)\|_{B \to B} = 0$.
Therefore the operator $R(x)$ is compact.
\end{proof}

\begin{thm} \label{thm:main}
For each fixed $x \in [0, \pi]$, 
the operator $I + \tilde R(x)$ has a bounded inverse operator, 
and equation \eqref{main} is uniquely solvable in the Banach space $B$.
\end{thm}

\begin{proof}
It follows from \eqref{smeqR}, that for each fixed $x \in [0, \pi]$, $(I - R(x)) (I + \tilde R(x)) = I$. 
Symmetrically, one gets $(I + \tilde R(x))(I - R(x)) = I$. 
Hence the operator $(I + \tilde R(x))^{-1}$ exists, and it is a linear bounded operator by Lemma~\ref{lem:Rbound}.
\end{proof}

Equation \eqref{main} is called {\it the main equation} of Inverse Problem 1.
Theorem~\ref{thm:main} together with Lemmas~\ref{lem:asymptrho}, \ref{lem:asymptalpha1}, 
\ref{lem:asymptalpha2} and \ref{lem:ranks} gives the necessity part in Theorem~\ref{thm:1}.

Now turn to the problem $L^*$, defined in \eqref{defL*}. 
Take the model problem $\tilde L^*  = L^*(\tilde Q(x), \tilde h, \tilde H)$
with the same potential $\tilde Q$ as the problem $\tilde L$ has. 
By virtue of \eqref{eqM}, the problems $L$ and $L^*$ (similarly, $\tilde L$ and $\tilde L^*$)
have the same spectral data. Symmetrically to \eqref{mainphi}, we obtain the relations 
\begin{equation} \label{mainphi*}
    \tilde \vv^*_{nqi}(x) = \vv^*_{nqi}(x) + \sum_{k = 0}^{\iy}
    \sum_{l = 1}^m \left(\tilde D(x, \la_{kl0}, \la_{nqi}) \al'_{kl0} \vv^*_{kl0}(x) -
     \tilde D(x, \la_{kl1}, \la_{nqi}) \al'_{kl1} \vv^*_{kl1}(x)\right),                                                    
\end{equation}
for each fixed $x \in [0, \pi]$, $n \ge 0$, $q = \overline{1, m}$, $i = 0, 1$.
Similarly to $B$, introduce the Banach space $B^*$ of column vectors. Then
$\psi^* = \{ \psi^*_k \}_{k = 0}^{\infty}$, $\psi_k^* = [\vv^*_{klj}(x)]_{l = \overline{1, m}, j = 0, 1}$
satisfy the linear equation
\begin{equation} \label{main*}
    (I + \tilde R^*(x)) \psi^*(x) = \tilde \psi^*(x)
\end{equation}
in $B^*$ for each fixed $x \in [0, \pi]$. Here $\tilde \psi^*(x)$ and $\tilde R^*(x)$ 
are constructed symmetrically to $\tilde \psi(x)$
and $\tilde R(x)$ by the model problem $\tilde L^*$ and the spectral data $\Lambda$, $\tilde \Lambda$.

\begin{lem} \label{lem:main*}
For each fixed $x \in [0, \pi]$, equation \eqref{main*} is uniquely solvable in the Banach space $B^*$
if and only if equation \eqref{main} is uniquely solvable.
\end{lem}

\begin{proof}
Fix $x \in [0, \pi]$. 
In view of Lemma~\ref{lem:Rbound}, the operators $\tilde R(x)$ and $\tilde R^*(x)$
are compact in the corresponding Banach spaces. Therefore it is sufficient to consider
homogeneous equations $\gamma(x) (I + \tilde R(x)) = 0$ and $(I + \tilde R^*(x)) \gamma^*(x) = 0$.
Let us prove only the ``if'' part, since the ``only if'' part can be proved symmetrically.

Suppose the equation $\gamma(x) (I + \tilde R(x)) = 0$ is uniquely solvable. Then there exists
a bounded inverse operator $\tilde P(x) = (I + \tilde R(x))^{-1}$ of the following form
$$
    (\psi \tilde P(x))_{nqi} = \sum_{k = 0}^{\infty} \sum_{l = 1}^m 
    \left( \psi_{kl0} \tilde P_{kl0, nqi}(x) - \psi_{kl1} \tilde P_{kl1, nqi}(x) \right),    
$$
$$
 \tilde P(x) = [\tilde P_{k, n}(x)]_{k, n \ge 0} = [\tilde P_{klj, nqi}(x)], \quad
 \psi = [\psi_{klj}] \in B, \quad n, k \ge 0 \: q, l = \overline{1, m}, \: i, j = 0, 1.
$$
%If there are multiple equal eigenvalues $\la_{klj}$, the corresponding components $\psi_{klj}$ are
%equal, and we may assume, that only one of the corresponding componemts $\tilde P_{klj, nqi}(x)$
%is different from zero. 
It follows from $(\psi \tilde P(x)) \in B$, that
\begin{equation} \label{estP}
    \| \tilde P_{rst, nqi}(x) - \tilde P_{rst, nlj}(x) \| \le C |\rho_{nqi} - \rho_{nlj}|, \quad
    r, n \ge 0, \: s, q, l = \overline{1, m}, \: t, i, j = 0, 1. 
\end{equation}

For simplicity, assume that all the values $\{ \la_{nqi}\}$ are distinct 
(the general case requires minor modifications).
The relation $\tilde P(x) (I + \tilde R(x)) = I$ yields
\begin{equation} \label{relPR}
    \tilde P_{rst, nqi}(x) + \sum_{k = 0}^{\iy} \sum_{l = 1}^n 
    \left( \tilde P_{rst, kl0}(x) \al'_{kl0} \tilde D(x, \la_{nqi}, \la_{kl0}) -
    \tilde P_{rst, kl1}(x) \al'_{kl1} \tilde D(x, \la_{nqi}, \la_{kl1})\right) = \de_{rst, nqi},
\end{equation}
$$
    n, r \ge 0, \quad s, q = \overline{1, m}, \quad t, i = 0, 1,
$$
where $\de_{rst, nqi} = I_m$, if $(r, s, t) = (n, q, i)$, and
$\de_{rst, nqi} = 0_m$ otherwise.

Let $\ga^*(x) = [\ga^*_{nqi}(x)]$ be a solution of the equation $(I + \tilde R^*(x)) \ga^*(x) = 0$:
\begin{equation} \label{eqvga*}
    \ga^*_{nqi}(x) + \sum_{k = 0}^{\iy} \sum_{l = 1}^m \left(
    \tilde D(x, \la_{kl0}, \la_{nqi}) \al'_{kl0} \ga^*_{kl0}(x) -
     \tilde D(x, \la_{kl1}, \la_{nqi}) \al'_{kl1} \ga^*_{kl1}(x) \right) = 0_m,
\end{equation}
$n \ge 0$, $q = \overline{1, m}$, $i = 0, 1$.
Then
\begin{multline*}
\sum_{n = 0}^{\iy} \sum_{q = 1}^m \sum_{i = 0}^1 (-1)^i \tilde P_{rst, nqi}(x) \al'_{nqi} \ga^*_{nqi}(x) \\ + 
\sum_{n, k = 0}^{\iy} \sum_{q, l = 1}^m \sum_{i, j = 0}^1 (-1)^{i + j} \tilde P_{rst, nqi}(x) \al'_{nqi}
\tilde D(x, \la_{klj}, \la_{nqi}) \al'_{klj} \ga^*_{klj}(x) = 0_m, \quad r \ge 0, \: s = \overline{1, m}, \: t = 0, 1.
\end{multline*}
Convergence of the series can be proved with help of \eqref{estP}. Using \eqref{relPR}, we obtain
$$
\sum_{n = 0}^{\iy} \sum_{q = 1}^m \sum_{i = 0}^1 (-1)^i \tilde P_{rst, nqi}(x) \al'_{nqi} \ga^*_{nqi}(x) +
\sum_{k = 0}^{\iy} \sum_{l = 1}^m \sum_{j = 0}^1 (-1)^j (\de_{rst, klj} - \tilde P_{rst, klj}(x)) \al'_{klj} \ga^*_{klj}(x) = 0_m.    
$$
Consequently, $\al'_{klj} \ga^*_{klj}(x) = 0_m$ for all $k \ge 0$, $l = \overline{1, m}$, $j = 0, 1$.
In view of \eqref{eqvga*}, we conclude that $\ga^*(x) = 0$,
so the homogeneous equation $(I + \tilde R^*(x)) \ga^*(x) = 0$ is uniquely solvable.

\end{proof}

Note that we do not use in the proof of Lemma~\ref{lem:main*} the fact, that $\Lambda$ is the spectral data of $L$,
but use only properties (A) and (R). Therefore this lemma can be used in the sufficiency part.

The main equation gives us a constructive solution of Inverse problem~1.
Solving \eqref{main}, we find the vector $\psi(x)$, i.e. the matrix-functions $\vv_{nqi}(x)$.

Denote
\begin{equation} \label{defeps}
     \ee_0(x) = \sum_{k = 0}^{\iy} \sum_{l = 1}^m \left(\vv_{kl0}(x) \alpha'_{kl0} \tilde \vv^*_{kl0}(x) -
     \vv_{kl1}(x) \alpha'_{kl1} \tilde \vv^*_{kl1}(x) \right),
     \quad \ee(x) = -2 \ee_0'(x).                                                        
\end{equation}

\begin{lem}
The series in \eqref{defeps} converges absolutely and uniformly on $[0, \pi]$,
the function $\ee_0(x)$ is absolutely continuous, and $\ee(x) \in L_2((0, \pi), \mathbb{C}^{m \times m})$. 
\end{lem}

\begin{proof}
Here we use ideas similar to the proof of Lemma~\ref{lem:Rbound}. Group the terms of \eqref{defeps}
in the following way:
\begin{multline} \label{smeqeps}
\sum_{l = 1}^m \left(\vv_{kl0}(x) \alpha'_{kl0} \tilde \vv^*_{kl0}(x) -
     \vv_{kl1}(x) \alpha'_{kl1} \tilde \vv^*_{kl1}(x) \right) = 
\sum_{l = 1}^m (\vv_{kl0}(x) - \vv_{kl1}(x)) \al'_{kl0} \tilde \vv^*_{kl0}(x) \\ +
\sum_{l = 1}^m \vv_{kl1}(x) \al'_{kl0} (\tilde \vv^*_{kl0}(x) - \tilde \vv^*_{kl1}(x)) +
\sum_{s = 1}^p \sum_{l \in J_s} (\vv_{kl1}(x) - \vv_{k m_s 1}(x))(\al'_{kl0} - \al'_{kl1}) \tilde \vv^*_{kl1}(x) \\ +
\sum_{s = 1}^p \sum_{l \in J_s} \vv_{k m_s 1}(x) (\al'_{kl0} - \al'_{kl1})(\tilde \vv^*_{kl1}(x) - \tilde \vv^*_{k m_s 1}(x)) +
\sum_{s = 1}^p (\vv_{k m_s 1}(x) - \vv_{k11}(x)) (\al_k^{(s)} - \tilde \al_k^{(s)}) \tilde \vv^*_{k m_s 1}(x) \\ +
\sum_{s = 1}^p \vv_{k11}(x) (\al_k^{(s)} - \tilde \al_k^{(s)}) (\tilde \vv^*_{k m_s 1}(x) - \tilde \vv^*_{k11}(x)) + 
\vv_{k11}(x) (\al_k - \tilde \al_k) \tilde \vv^*_{k11}(x).
\end{multline}
It follows from \eqref{defxi}, \eqref{defOmega} and Lemma~\ref{lem:Schwarz}, that
the series in \eqref{defeps} converges absolutely and uniformly on $[0, \pi]$: 
$$
    \| \ee_0(x) \| \le C \sum_{k = 0}^{\infty} \xi_k < \iy.
$$
Let us analyze the derivative of the first term in \eqref{smeqeps}:
\begin{multline*}
    S'(x) := \frac{d}{dx} \Bigl( (\vv_{kl0}(x) - \vv_{kl1}(x)) \al'_{kl0} \tilde \vv^*_{kl0}(x) \Bigr) = \\
    (\vv'_{kl0}(x) - \vv'_{kl1}(x)) \al'_{kl0} \tilde \vv^*_{kl0}(x) +
    (\vv_{kl0}(x) - \vv_{kl1}(x)) \al'_{kl0} \tilde {\vv^*}'_{kl0}(x)
\end{multline*}
The other terms can be treated similarly.
Using asymptotics \eqref{asymptrho}, \eqref{asymptphi} and \eqref{asymptdphi} together with Schwarz's lemma, one gets
\begin{gather*}
  \vv'_{kl0}(x) - \vv'_{kl1}(x) = - \cos k x \gamma_{kl} x I_m + \frac{K_k(x)}{k + 1}, \quad \tilde \vv^*_{kl0}(x) = \cos k x + O(k^{-1}),  \\
  \vv_{kl0}(x) - \vv_{kl1}(x) = - \sin k x \frac{\gamma_{kl}}{k + 1} x I_m + \frac{K_k(x)}{(k + 1)^2}, \quad 
  \tilde {\vv^*}'_{kl0}(x) = - k \sin k x + O(1),
\end{gather*}
where $\gamma_{kl} = (k + 1)(\rho_{kl0} - \rho_{kl1})$, $\{ \ga_{kl} \} \in l_2$, 
$K_k(x)$ denote various sequences of continuous on $[0, \pi]$ matrix functions, such that
$\{ \max\limits_x \| K_k(x)  \| \} \in l_2$. Then
$$
    S'(x) = - \cos 2 k x \ga_{kl} x \al'_{kl0} + \frac{K_k(x)}{k + 1}.
$$
By the Riesz-Fischer theorem, 
$$
    x \sum_{k = 0}^{\infty} \cos 2 k x \ga_{kl} \al'_{kl0} \in L_2((0, \pi), \mathbb{C}^{m \times m}).
$$
The series $\sum\limits_{k = 0}^{\infty} \dfrac{K_k(x)}{k + 1}$ converges absolutely and uniformly with respect to
$x \in [0, \pi]$. Hence $\ee(x) \in L_2((0, \pi), \mathbb{C}^{m \times m})$.
\end{proof}

The next lemma gives formulas for recovering 
the potential $Q(x)$ and the coefficients of the boundary conditions $h$
and $H$. It can be proved similarly to \cite[Lemma~8]{Bond11}.

\begin{lem}
The following relations hold
\begin{equation} \label{relQ}
    Q(x) = \tilde Q(x) + \ee(x), \quad
    h = \tilde h - \ee_0(0), \quad H = \tilde H + \ee_0(\pi),                           
\end{equation}
\end{lem}

Thus, we obtain the following algorithm for the solution of
Inverse Problem~1.

\medskip

{\bf Algorithm 1.} Given the data $\Lambda.$
\begin{enumerate}
\item Choose $\tilde L \in A(\om)$, and calculate $\tilde \psi(x)$
and $\tilde R(x).$
\item Find $\psi(x)$ by solving equation~\eqref{main}, and calculate
$\vv_{nqi}(x).$
\item Construct $Q(x)$, $h$ and $H$ by~\eqref{relQ}.
\end{enumerate}

\medskip

{\large \bf 5. Conditions (M), (E), (C) and (PW)}\\

In this section, we establish the connection between the solvability of the main equation (M),
the condition (E) and the completeness of some system of functions (C). The condition (C)
and its equivalent reformulation (PW) will be given further in this section.

Let $\om \in \mathcal{D}$ and data $\Lambda \in \mbox{Sp}$ satisfy conditions (A) and (R) of Theorem~\ref{thm:1}.
Let $\tilde L$ be an arbitrary problem from the class $A_{1,2}(\om)$.

\begin{lem} \label{lem:EM}
(E) follows from (M).
\end{lem}

\begin{proof}
Let $\ga(\la)$ be a row vector, 
entire in $\la$ and satisfying the relations 
$\ga(\la) = O(\exp(|\tau|\pi))$, $\ga(\la_{nq0}) \al_{nq0} = 0$ for all $n \ge 0$, $q = \overline{1, m}$.

Schwarz's lemma together with asymptotics \eqref{asymptrho} yields
\begin{equation} \label{estgamma1}
\| \ga(\la_{nqi}) - \ga(\la_{nlj}) \| \le C |\rho_{nqi} - \rho_{nlj}|, \quad
n \ge 0, \: q, l = \overline{1, m}, \: i, j = 0, 1.
\end{equation} 

Consider the function
\begin{equation} \label{deftgamma}
    \tilde \ga(\la) := \ga(\la) + \sum_{k = 0}^{\iy} \sum_{l = 1}^{\iy} \Bigl[
    \ga(\la_{kl0}) \al'_{kl0} \tilde D(\pi, \la, \la_{kl0}) - 
    \ga(\la_{kl1}) \al'_{kl1} \tilde D(\pi, \la, \la_{kl1})
    \Bigr]
\end{equation}
In order to prove the convergence of the series in \eqref{deftgamma}, we apply to following formal transformation
\begin{multline} \label{transformgamma}
\tilde \ga(\la) = \ga(\la) + \sum_{k = 0}^{\iy} \Biggl[
\sum_{l = 1}^m (\ga(\la_{kl0}) - \ga(\la_{kl1})) \al'_{kl0} \tilde D(\pi, \la, \la_{kl0})  +
\sum_{l = 1}^m \ga(\la_{kl1}) \al'_{kl0} (\tilde D(\pi, \la, \la_{kl0}) \\ - \tilde D(\pi, \la, \la_{kl1})) +
\sum_{s = 1}^p \sum_{l \in J_s} (\ga(\la_{kl1}) - \ga(\la_{k m_s 1}))(\al'_{kl0} - \al'_{kl1}) \tilde D(\pi, \la, \la_{kl1})  +
\sum_{s = 1}^p \sum_{l \in J_s} \ga(\la_{k m_s 1}) \\ \times (\al'_{kl0} - \al'_{kl1})(\tilde D(\pi, \la, \la_{kl1}) - \tilde D(\pi, \la, \la_{k m_s 1}))  +
\sum_{s = 1}^p (\ga(\la_{k m_s 1}) - \ga(\la_{k11}) (\al_k^{(s)} - \tilde \al_k^{(s)}) \tilde D(\pi, \la, \la_{k m_s 1}) \\ +
\sum_{s = 1}^p \ga(\la_{k11}) (\al_k^{(s)} - \tilde \al_k^{(s)}) (\tilde D(\pi, \la, \la_{k m_s 1}) - \tilde D(\pi, \la, \la_{k11})) + 
\ga(\la_{k11}) (\al_k - \tilde \al_k) \tilde D(\pi, \la, \la_{k11}) \Biggr].
\end{multline}
By virtue of \eqref{defxi}, \eqref{estgamma1} and the estimates
\begin{multline*}
    \| \tilde D(\pi, \la, \la_{klj}) \| \le C \exp(|\tau|\pi), \quad
    \| \tilde D(\pi, \la, \la_{klj}) - \tilde D(x, \la, \la_{kqi} \| \le C|\rho_{klj} - \rho_{kqi}|\exp(|\tau|\pi), \\
    \quad \mbox{Re} \, \rho \geq 0, \: k \ge 0, \: l, q = \overline{1, m}, \: i, j = 0, 1,
\end{multline*}
we get
$$
   \|\tilde \gamma(\la) \| \le \| \ga(\la) \| + C \exp(|\tau|\pi) \sum_{k = 0}^{\iy} \xi_k.
$$
Taking \eqref{defOmega} into account, we conclude that the series in \eqref{deftgamma} converges
to an entire function, satisfying the estimate $\tilde \ga(\la) = O(\exp(|\tau|\pi))$.

Substitute $\la = \la_{nq1}$ into \eqref{deftgamma} and multiply the result
by $\al_{nq1}$. Definition \eqref{defD} and Lemma~\ref{lem:sym} yield
\begin{equation} \label{sym}
\al'_{kl1} \tilde D(\pi, \la_{nq1}, \la_{kl1}) \al_{nq1} = \begin{cases}
                                                         \al_{nq1}, \quad \text{if} \: \al_{nq1} = \al_{kl1} \: 
                                                         \text{and} \: \al'_{kl1} \ne 0_m, \\
                                                         0_m, \quad \text{otherwise.}
                                                      \end{cases}
\end{equation}
Therefore, we obtain $\tilde \ga(\la_{nq1}) \al_{nq1} = 0$ for all $n \ge 0$, $q = \overline{1, m}$.
Thus, we have got the entire function $\tilde \ga(\la)$, satisfying the presupposition of (E) for 
{\it the spectral data $\tilde \Lambda$ of the model problem} $\tilde L$. But (E) holds for $\tilde \Lambda$
by~Lemma~\ref{lem:E}. Hence $\tilde \ga(\la) \equiv 0$, and by \eqref{deftgamma}
$$
    \ga(\la_{nqi}) + \sum_{k = 0}^{\infty} \sum_{l = 1}^m \left( \ga(\la_{kl0}) \al'_{kl0}
    \tilde D(\pi, \la_{nqi}, \la_{kl0}) - \ga(\la_{kl1}) \al'_{kl1} \tilde D(\pi, \la_{nqi}, \la_{kl1}) \right) = 0.
$$
We see that the vector $\psi = [\ga(\la_{nqi})] \in B$ satisfy the homogeneous main equation
$\psi (I + \tilde R(\pi)) = 0$. It follows from (M), that $\ga(\la_{nqi}) = 0$ for all $n \ge 0$, $q = \overline{1, m}$,
$i = 0, 1$. Using \eqref{deftgamma} once again, we arrive at $\ga(\la) \equiv 0$. Thus, we have proved (E).

\end{proof}

Introduce the subspaces $\mathcal{E}_{nq} = \mbox{Ran} \, \al'_{nq} = \{\mathcal{E} = \al'_{nq} h, \, h \in \mathbb{C}^m \}$.
Note that we intendently use $\al'_{nq}$ instead of $\al_{nq}$ in this definition, in order not to include
the same subspaces multiple times. Let $\left\{ \mathcal{E}_{nq}^{(i)} \right\}_{i = 1}^{m_{nq}}$ be an ortonormal basis of 
$\mathcal{E}_{nq}$. The number $m_{nq}$ coincide with the multiplicity of the corresponding eigenvalue $\la_{nq}$, 
if $\al'_{nq} \ne 0_m$, and $\mathcal{E}_{nq} = \varnothing$ otherwise.

\begin{lem} \label{lem:complete}
Let $\vv(x, \la)$ be an arbitrary matrix-function, continuous with respect to $x \in [0, \pi]$ and 
entire in $\la$, satisfying the asymptotic relation \eqref{asymptphi}.
Suppose that (E) holds. Then the system 
\begin{equation} \label{vvsys}
    \vv(x, \la_{nq}) \mathcal{E}_{nq}^{(i)}, \quad n \ge 0, \: q = \overline{1, m}, \: i = \overline{1, m_{nq}},
\end{equation}
is complete in $L_2((0, \pi), \mathbb{C}^{m})$.
\end{lem}

\begin{proof}
Consider a vector-function $f(x) \in L_2((0, \pi), \mathbb{C}^m)$, such that
$$
    \int_0^{\pi} f^{\dagger}(x) \vv(x, \la_{nq}) \mathcal{E}_{nq}^{(i)} \, dx = 0
$$
for all $n \ge 0$, $q = \overline{1, m}$, $i = \overline{1, m_{nq}}$.
It is easy to check that the function
$$
    \ga(\la) = \int_0^{\pi} f^{\dagger}(x) \vv(x, \la) \, dx
$$ 
satisfy all the properties in (E) and, consequently, equals zero. 
Therefore $f(x) = 0$ and the system \eqref{vvsys} is complete.
\end{proof}

\begin{lem} \label{lem:basis}
Let $\vv(x, \la)$ satisfy conditions of Lemma~\ref{lem:complete}, and the system~\eqref{vvsys} is complete.
Then the system \eqref{vvsys} is a basis in $L_2((0, \pi), \mathbb{C}^m)$.

\end{lem}

\begin{proof}
The basis property of the system~\eqref{vvsys} follows from its completeness and $l_2$-closeness to the
ortonormal basis
\begin{equation} \label{cossys}
    \cos nx \mathcal{E}_{nq}^{(i)}, \quad n \ge 0, \: q = \overline{1, m}, \: i = \overline{1, m_{nq}},
\end{equation}
According to asymptotics \eqref{asymptrho} and \eqref{asymptphi}, 
$$
    \cos nx I_m - \vv(x, \la_{nq}) = O(n^{-1}), \quad \left\{ \| (\cos nx I_m - \vv(x, \la_{nq})) \mathcal{E}_{nq}^{(i)} \| \right\} \in l_2.
$$
In order to prove the basis property for the system~\eqref{cossys}, it is sufficient to show, 
that the system $\left\{ \mathcal{E}_{nq}^{(i)} \right \}$ is a basis in $\mathbb{C}^m$
for a fixed $n$, for all sufficiently large values of $n \ge N$.
If the elements of \eqref{cossys} for $n < N$ are linearly dependent, one can change them to
$\cos nx e_q$, $q = \overline{1, m}$, where $\{ e^q \}_{q = 1}^m$ is the standard coordinate basis.
Since (R) holds, i.e. ranks of the weight matrices equal multiplicities of the corresponding eigenvalues,
the total number of the vectors $\left\{ \mathcal{E}_{nq}^{(i)} \right \}$ for a fixed $n$ is $m$.
Suppose there exists a vector $h$ ortogonal to all $ \mathcal{E}_{nq}^{(i)}$. Then
$h^{\dagger}\al_{nq} = 0$ for all $q = \overline{1, m}$, and $h^{\dagger} \al_n = 0$.
But in view of \eqref{asymptalpha3}, $\det \al_n \ne 0$ for sufficiently large values of $n$.
Thus, the considered system of vectors is a basis. 
\end{proof}

Similar facts can be obtained for the problem $L^*$. 
Let $\mathcal{E}^*_{nq} = \{\mathcal{E}^* = h \al'_{nq}, \, h \in \mathbb{C}^{m, T} \}$.
Denote by $\left\{ \mathcal{E}^{*,(i)}_{nq} \right\}_{i = 1}^{m_{nq}}$ an ortonormal basis of 
$\mathcal{E}^*_{nq}$, consisting of row vectors. The following lemma summarizes results, similar to Lemmas~\ref{lem:EM}-\ref{lem:basis}.
In fact, the solvability the main equation \eqref{main*} for $L^*$, instead of \eqref{main}, can be used to prove the lemma. 

\begin{lem} \label{lem:basis*}
Assume that $\vv(x, \la)$ is an arbitrary matrix-function, continuous with respect to $x \in [0, \pi]$ and 
entire in $\la$, satisfying the asymptotic relation \eqref{asymptphi},
and (M) holds. Then the system 
$$
    \mathcal{E}_{nq}^{(i), *} \vv(x, \la_{nq}), \quad n \ge 0, \: q = \overline{1, m}, \: i = \overline{1, m_{nq}},
$$
is a basis in $L_2((0, \pi), \mathbb{C}^{m, T})$.
\end{lem} 

By virtue of Lemma~\ref{lem:complete}, the following condition: 

{\it
(C) The system of functions
$$
    \cos \rho_{nq}x \mathcal{E}_{nq}^{(i)}, \quad n \ge 0, \: q = \overline{1, m}, \: i = \overline{1, m_{nq}},
$$
is complete in $L_2((0, \pi), \mathbb{C}^m)$.
}

follows from (E).
The converse is not true.

Indeed, there is the one-to-one correspondence between vector-functions 
$f(x) \in L_2((0, \pi), \mathbb{C}^m)$ and row vectors
$\ga(\la) = \int_0^{\pi} f^{\dagger}(x) \cos \rho x \, dx$ of an even Paley-Wiener class $\mbox{PW}$, 
defined by the following conditions:
\begin{enumerate}
\item $\ga(\la)$ is entire,
\item $\ga(\la) = O(\exp(|\tau|\pi))$,
\item $\int_0^{\infty}|\ga(\rho^2)|^2 \, d\rho < \infty$.
\end{enumerate}
Therefore, condition (C) is equivalent to the following condition: 

{\it (PW) For any row vector $\gamma(\la) \in \mbox{PW}$, if $\gamma(\la_{nq})\alpha_{nq} = 0$ for all 
$n \geq 0$, $q = \overline{1, m}$, then $\gamma(\la) \equiv 0$.
}

\bigskip

{\large \bf 6. Proof of Theorem 1}\\

The necessity in Theorem~\ref{thm:1} is contained in Lemmas~\ref{lem:asymptrho}, \ref{lem:asymptalpha1},
\ref{lem:asymptalpha2}, \ref{lem:ranks} and Theorem~\ref{thm:main}.

Turn to the proof of the sufficiency.
Let data $\{\la_{nq}, \alpha_{nq}\}_{n \geq 0, q =
\overline{1, m}} \in \mbox{Sp}$ be given. Choose $\tilde L \in A_{1,2}(\om)$ and construct
$\tilde \psi(x)$, $\tilde R(x)$. Assume that the conditions of Theorem~\ref{thm:1} hold.
Let $\psi(x) = \{\psi_n(x)\}_{n \ge 0} \in B$ be the unique solution of the main equation~\eqref{main}.
The proofs of Lemmas~\ref{lem:estpsi}-\ref{lem:lphi} are analogous
to ones described in \cite[Sec.~1.6.2]{FY01}.

\begin{lem} \label{lem:estpsi}
For $n \geq 0$, the functions $\psi_n(x)$ are continuously differentiable with respect to $x$ on $[0, \pi]$, 
and the following relations hold
$$
\begin{array}{c}
   \| \psi^{(\nu)}_{n}(x)\|_{B(G_n)} \leq C(n + 1)^{\nu}, \quad \nu = 0, 1, \quad x \in [0, \pi], \\
   \|\psi_{n}(x) - \tilde  \psi_{n}(x)\|_{B(G_n)} \leq C \Omega \eta_n, \quad
   \| \psi'_{n}(x) - \tilde \psi'_{n}(x) \|_{B(G_n)} \leq C \Omega, \quad x \in [0, \pi].
\end{array}
$$
where
$$
    \eta_n := \left(\sum_{k = 0}^{\iy} \frac{1}{(k + 1)^2(|n - k| + 1)^2} \right)^{1/2}.
$$
\end{lem} 

By virtue of Lemma~\ref{lem:estpsi}, the matrix-functions $\vv_{nqi}(x) := \psi_n(x, \rho_{nqi})$
satisfy the following estimates
\begin{equation} \label{estphin}
    \begin{array}{c}
    \| \vv_{nqi}^{(\nu)}(x)\| \leq C(n + 1)^{\nu}, \quad \nu = 0, 1, \\
    \| \vv_{nqi}(x) - \tilde \vv_{nqi}(x)\| \leq C \Omega \eta_n, \quad
    \| \vv'_{nqi}(x) - \tilde \vv'_{nqi}(x)\| \leq C \Omega, \quad q = \overline{1, m}, \\
    \| \vv_{n q i}(x) - \vv_{n l j}(x) \| \leq C |\rho_{nqi} - \rho_{nlj}|, \quad
     q, l = \overline{1, m}, \: i, j = 0, 1.
     \end{array}                                                                              
\end{equation}

Further, we construct the matrix-functions $\vv(x, \la)$ and
$\Phi(x, \la)$ by the formulas
\begin{equation} \label{defphinew}
    \vv(x, \la) = \tilde \vv(x, \la) - \sum_{k = 0}^{\iy} \sum_{l = 1}^m \sum_{j = 0}^1 (-1)^j
    \vv_{klj}(x) \alpha'_{klj} \frac{\langle \tilde \vv^*_{klj}(x),  \tilde \vv(x, \la) \rangle}{\la - \la_{klj}},
\end{equation}
\begin{equation} \label{defPhinew}
    \Phi(x, \la) = \tilde \Phi(x, \la) - \sum_{k = 0}^{\iy} \sum_{l = 1}^m \sum_{j = 0}^1 (-1)^j
    \vv_{klj}(x) \alpha'_{klj} \frac{\langle \tilde \vv^*_{klj}(x),  \tilde \Phi(x, \la) \rangle}{\la - \la_{klj}},
\end{equation}
(see \eqref{mainphila}, \eqref{mainPhi}) and the boundary value problem $L(Q(x), h, H)$ via \eqref{relQ}. Clearly,
$\vv(x, \la_{nqi}) = \vv_{nqi}(x)$.

Using estimates~\eqref{estphin}, one can show that the entries of $\ee_0(x)$ are
absolutely continuous and the entries of
$\ee(x)$ belong to $L_2(0, \pi)$.
Consequently, we get

\begin{lem}
$Q(x) \in L_2((0, \pi), \mathbb{C}^{m \times m})$.
\end{lem}

\begin{lem} \label{lem:lphi}
The following relations hold
$$
   \ell \vv_{nqi}(x) = \la_{nqi} \vv_{nqi}(x), \quad
   \ell \vv(x, \la) = \la \vv(x, \la), \quad \ell \Phi(x, \la) = \la \Phi(x, \la),
$$
$$
   \vv(0, \la) = I_m, \quad \vv'(0, \la) = h, \quad U(\Phi) = I_m, \quad V(\Phi) = 0_m.
$$
\end{lem}

\begin{proof}
Let us prove only the relation $V(\Phi) = 0_m$, since other ones can be obtained similarly to the scalar case
\cite{FY01}.
It follows from \eqref{relQ} and \eqref{defphinew}, that
$$
    \tilde V(\tilde \vv) = V(\vv) + \sum_{k = 0}^{\iy} \sum_{l = 1}^m \left( V(\vv_{kl0}) \al'_{kl0} \tilde D(\pi, \la, \la_{kl0}) - 
    V(\vv_{kl1}) \al'_{kl1} \tilde D(\pi, \la, \la_{kl1}) \right),
$$
\begin{multline*}
    \tilde V(\tilde \vv_{nq1}) \al_{nq1} = V(\vv_{nq1}) \al_{nq1} + \sum_{k = 0}^{\iy} \sum_{l = 1}^m 
    \Bigl( V(\vv_{kl0}) \al'_{kl0} \tilde D(\pi, \la_{nq1}, \la_{kl0}) \al_{nq1} \\ - 
    V(\vv_{kl1}) \al'_{kl1} \tilde D(\pi, \la_{nq1}, \la_{kl1}) \al_{nq1} \Bigr).
\end{multline*}
Using \eqref{prodVal} and \eqref{sym}, we derive
$$
    \sum_{k = 0}^{\infty} \sum_{l = 1}^m V(\vv_{kl0}) \al'_{kl0} \tilde D(\pi, \la_{nq1}, \la_{kl0}) \al_{nq1} = 0.
$$
Taking \eqref{defD} into account, we obtain
\begin{equation} \label{smeq7}
    \int_0^{\pi} f(x) \tilde \vv_{nq1}(x) \al_{nq1} \, dx = 0_m, \quad n \ge 0, \: q = \overline{1, m}, 
\end{equation}
\begin{equation} \label{smeq6}
f(x) := \sum_{k = 0}^{\iy} \sum_{l = 1}^m V(\vv_{kl0}) \al'_{kl0} \tilde \vv^*_{kl0}(x),
\end{equation}

Let us prove that $f \in L_2((0, \pi), \mathbb{C}^{m \times m})$.
Indeed, $\vv(x, \la)$ is a solution of equation \eqref{eqv}, therefore the relation \eqref{asymptVphi} holds and
$$
    V(\vv_{kl0}) = (-1)^k (\om - \om_l I_m) + K_k.
$$
By virtue of \eqref{asymptalpha2}, $\left\{ \| V(\vv_{kl0}) \al'_{kl0} \|\right\} \in l_2$.
Similarly to asymptotic relation \eqref{asymptphi}, we get $\tilde \vv^*_{kl0}(x) = \cos k x I_m + O(k^{-1})$,
and the series in \eqref{smeq7} converges in $L_2$. Hence $f(x)$ belongs to $L_2$.

The system of linearly independent columns of the matrices $\vv_{kl1}(x) \al_{kl1}$ 
is complete in $L_2((0, \pi), \mathbb{C}^m)$ by Lemma~\ref{lem:complete}. Hence, it follows from \eqref{smeq7}, that $f(x) \equiv 0$,

\begin{comment}
Let $P^*_{kl0}$ be an ortogonal projector to the linear span of the rows of $\al'_{kl0}$: $\mathcal{E}^*_{kl0} = \{\mathcal{E}^* = h \al'_{kl0}, \, h \in \mathbb{C}^{m, T} \}$,
$g P^*_{kl0} \in \mathcal{E}^*_{kl0}$, for every row vector $g \in C^{m, T}$.
Clearly, $P_{kl0} \al'_{kl0} = \al'_{kl0}$, and
$$
    f(x) := \sum_{k = 0}^{\iy} \sum_{l = 1}^m (V(\vv_{kl0}) \al'_{kl0}) P_{kl0}^* \tilde \vv^*_{kl0}(x),
$$
\end{comment}

By Lemma~\ref{lem:basis*}, 
linearly independent rows of the matrices $\al'_{kl0} \vv^*_{kl0}(x)$ form a basis in $L_2((0, \pi), \mathbb{C}^{m, T})$.
Therefore, the terms in \eqref{smeq6} can not differ from zero. Hence
\begin{equation} \label{smeq8}
    V(\vv_{kl0}) \al'_{kl0} = 0_m,  \quad k \ge 0, \quad l = \overline{1, m}.
\end{equation}

\begin{comment}
Hence
\begin{multline*}
    f(x) = \sum_{k = 0}^{\iy} (-1)^k \sum_{s = 1}^p (\om - \om_{m_s}I_m) \Biggl[ \sum_{l \in J_s} \al'_{kl0}
    (\tilde \vv^*_{kl0}(x) - \tilde \vv^*_{k m_s 0}(x)) + (I^{(s)} + K_k) \tilde \vv^*_{k m_s 0}(x) \Biggr] \\ +
    \sum_{k = 0}^{\iy} K_k \tilde \vv^*_{kl0}(x).    
\end{multline*}
Note that $(\om - \om_{m_s}I_m) I^{(s)} = 0_m$. Applying Lemma~\ref{lem:Schwarz}, \eqref{asymptrho} and \eqref{asymptphi}, 
we conclude that
$$
    f(x) = \sum_{k = 0}^{\infty} K_k \cos k x + \sum_{k = 0}^{\infty} \frac{\eta_k(x)}{k + 1}, \quad \{ \max_x \| \eta_k (x) \| \} \in l_2.
$$
The first series is a function from $L_2$ by virtue of Riesz-Fischer theorem, 
and the second one converges absolutely and uniformly with respect to $x \in [0, \pi]$ and gives a continuous matrix-function.
\end{comment}

Using \eqref{relQ} and \eqref{defPhinew}, we obtain
$$
    \tilde V(\tilde \Phi) = V(\Phi) + \sum_{k = 0}^{\infty} \sum_{l = 1}^m \Biggl( V(\vv_{kl0}) \al'_{kl0}
    \frac {\langle \tilde \vv_{kl0}^*(x), \tilde \Phi(x, \la) \rangle_{x = \pi}}{\la - \la_{kl0}} -  V(\vv_{kl1}) \al'_{kl1}
    \frac {\langle \tilde \vv_{kl1}^*(x), \tilde \Phi(x, \la) \rangle_{x = \pi}}{\la - \la_{kl1}} \Biggr) 
$$
Applying \eqref{smeq8} and the following relations:
$$
    \tilde V(\tilde \Phi) = 0_m, \quad 
    \al_{kl1} \frac {\langle \tilde \vv_{kl1}^*(x), \tilde \Phi(x, \la) \rangle_{x = \pi}}{\la - \la_{kl1}} =
    \al_{kl1} \frac{\tilde V^*(\tilde \vv_{kl1}^*) \tilde \Phi(\pi, \la) - \tilde \vv_{kl1}^*(\pi) \tilde V(\tilde \Phi)}{\la - \la_{kl1}} = 0_m,
$$
we conclude that $V(\Phi) = 0_m$.
\end{proof}

In order to finish the proof of Theorem~\ref{thm:1}, it remains to show that
the constructed boundary value problem $L(Q, h, H)$ belongs to $A_{1,2}(\om)$ and
the given data $\{\la_{nq}, \alpha_{nq} \}$ coincide with the
spectral data of $L$.
In view of Lemma~\ref{lem:lphi}, the matrix-function $\Phi(x, \la)$ is the
Weyl solution of $L$. Let us get the representation for the Weyl
matrix:
$$
    M(\la) = \Phi(0, \la) = \tilde M(\la) -  \sum_{k = 0}^{\iy} \sum_{l = 1}^m \sum_{j = 0}^1 (-1)^j \vv_{klj}(0) \alpha'_{klj}
    \frac{\langle \tilde \vv^*_{klj}(x), \tilde \Phi(x, \la) \rangle_{x = 0}}{ \la - \la_{klj}}
    \tilde M(\la)
$$
$$
    + \sum_{k = 0}^{\iy} \sum_{l = 1}^m \left(\frac{\alpha'_{kl0}}{\la - \la_{kl1}} - \frac{\alpha'_{kl1}}{\la - \la_{kl1}} \right).
$$
Using the equality (see \cite{Yur06}) 
$$
    \tilde M(\la) = \sum_{k = 0}^{\iy} \sum_{l = 1}^m \frac{\alpha'_{kl1}}{\la - \la_{kl1}},
$$
we arrive at
$$
     M(\la) = \sum_{k = 0}^{\iy} \sum_{l = 1}^m \frac{\alpha'_{kl0}}{\la - \la_{kl0}}.
$$
Consequently, $\{ \la_{kl0} \}$ are simple poles of the Weyl
matrix $M(\la)$, and $\{ \alpha_{kl0}\}$ are residues at the
poles. 
%Clearly, each eigenvalue $\la_{kl}$ has linear independent eigenfunctions $\vv(x, \la_{kl}) \mathcal{E}_{kl}$,
%and their number coincide with the rank of the corresponding weight matrix $\al_{kl}$. 
%It is easy to show that the problem $L$ can not have other eigenfunctions or associated functions. 
%Then Assumption~1 holds. Using \eqref{relQ}, we get $h + H + \frac{1}{2} \int_0^{\pi} Q(t)\,dt = \om$,
So $L \in A_{1,2}(\om)$, and $\Lambda$ is the spectral data of $L$.
Theorem~\ref{thm:1} is proved.

\bigskip

{\large \bf 7. The self-adjoint case}\\

Suppose $Q(x) = Q^{\dagger}(x)$ a.e. on $[0, \pi]$, $h = h^{\dagger}$, $H = H^{\dagger}$.
In this case, the matrix $\om = h + H + \frac{1}{2} \int_0^{\pi} Q(x) \, dx$ is Hermitian.
Therefore, one can diagonalize it, applying a unitary transform. Without loss of generality, we consider 
$L \in A(\om)$, $\om = \om^{\dagger} \in \mathcal{D}$.

The eigenvalues $\la_{nq}$ are real, and the poles of the matrix function $(V(\vv(x, \la))^{-1}$ are simple~\cite{Bond11}. 
Furthermore, it is easy to check that $\vv^*(x, \la) = \vv^{\dagger}(x, \bar \la)$,
$S^*(x, \la) = S^{\dagger}(x, \bar \la)$. Consequently, $M^*(\la) = M^{\dagger}(\bar \la) = M(\la)$
and $\al_{nq} = \al^{\dagger}_{nq}$, for all $n \ge 0$, $q = \overline{1, m}$.
By Lemma~\ref{lem:sym},
$$
    \al_{nq} = \al_{nq}^{\dagger} \int_0^{\pi} \vv^{\dagger}(x, \la_{nq}) \vv(x, \la_{nq}) \, dx \, \al_{nq} \ge 0.
$$
Taking the last fact together with asymptotics \eqref{asymptalpha1}, we conclude, that $\| \al_{nq}\| \le C$, 
$n \ge 0$, $q = \overline{1, m}$.

Thus, we have shown that Assumptions 1 and 2 holds automatically in the self-adjoint case. Moreover, we have proven 
(S) in Theorem~\ref{thm:2}. Note that (E) by necessity was proven in Lemma~\ref{lem:E}. 
So we have finished the proof of the necessity.

In order to prove the sufficiency  in Theorem~\ref{thm:2}, it remains to show that 
the solvability of the main equation \eqref{main} follows from (E) together with other conditions.

Let data $\{\la_{nq}, \alpha_{nq}\}_{n \geq 0, q =
\overline{1, m}} \in \mbox{Sp}$, satisfying the conditions of Theorem~\ref{thm:2}, be given. 
Choose a model problem $\tilde L \in A(\om)$, construct $\tilde \psi(x)$, $\tilde R(x)$, and consider the
main equation \eqref{main}.
 
\begin{lem} \label{lem:homo}
For each fixed $x \in [0, \pi]$, the operator $I + \tilde R(x)$, acting from $B$ to $B$,
has a bounded inverse operator, and the main equation~\eqref{main}
has a unique solution $\psi(x) \in B$.
\end{lem}

\begin{proof}
By Lemma~\ref{lem:Rbound} the operator $\tilde R(x)$ is compact.
Therefore it is sufficient to prove that the homogeneous equation
\begin{equation} \label{homo}
    \ga(x) (I + \tilde R(x)) = 0,                                                           
\end{equation}
where $\ga(x) \in B$,
has only the zero solution.
Let $\ga(x) = \{ \ga_n(x)\}_{n \ge 0} \in B$ be a solution of \eqref{homo}.
Denote $\ga_{nqi}(x) = \ga_n(x, \rho_{nqi})$. Then
\begin{equation} \label{relgamma}
    \ga_{nqi}(x) + \sum_{k = 0}^{\iy} \sum_{l = 1}^m \left(\ga_{kl0}(x) \al'_{kl0} \tilde D(x, \la_{nqi}, \la_{kl0}) -
    \ga_{kl1}(x) \al'_{kl1} \tilde D(x, \la_{nqi}, \la_{kl1})\right) = 0_m,
\end{equation}
and the following estimates are valid
\begin{equation} \label{estgamma}
\| \ga_{nqi}(x) \| \le C, \quad \| \ga_{nqi}(x) - \ga_{nlj}(x) \| \le C |\rho_{nqi} - \rho_{nlj}|, 
\quad n \ge 0, \: q, l = \overline{1, m}, \: i, j = 0, 1.
\end{equation}

Construct the matrix-functions $\gamma(x, \la)$, $\Gamma(x, \la)$
and $\mathscr{B}(x, \la)$ by the formulas
\begin{equation} \label{defgamma}
   \gamma(x, \la) = -\sum\limits_{k = 0}^{\iy} \sum\limits_{l = 1}^m \left(
   \gamma_{kl0}(x) \alpha'_{kl0} \tilde D(x, \la, \la_{kl0})
   - \gamma_{kl1}(x) \alpha'_{kl1} \tilde D(x, \la, \la_{kl1}) \right),                                                                                   
\end{equation}
\begin{equation} \label{defGamma}
    \Gamma(x, \la) = -\sum\limits_{k = 0}^{\iy} \sum\limits_{l = 1}^m \biggl[
    \gamma_{kl0}(x) \alpha'_{kl0} \frac{\langle \tilde \vv_{kl0}^*(x),
    \tilde \Phi(x, \la) \rangle}{\la - \la_{kl0}}
     - \gamma_{kl1}(x) \alpha'_{kl1} \frac{\langle \tilde \vv_{kl1}^*(x),
     \tilde \Phi(x, \la) \rangle}{\la - \la_{kl1}}
    \biggr],                                                                                  
\end{equation}
$$
     \mathscr{B}(x, \la) = \gamma^{\dagger}(x, \bar{\la}) \Gamma(x, \la).
$$
In view of~\eqref{defD}, the matrix-function $\gamma(x, \la)$ is entire in
$\la$ for each fixed $x$. 
The functions $\Gamma(x, \la)$ and $\mathscr{B}(x, \la)$ are meromorphic in $\la$ with simple poles $\la_{nqi}$.
According to \eqref{relgamma}, $\gamma(x, \la_{nqi}) = \gamma_{nqi}(x)$.
We calculate residues of $\mathscr{B}(x, \la)$ (for simplicity we assume that
$\{ \la_{nq0}\} \cap \{\la_{nq1} \} = \emptyset$):
$$
    \mathop{\mathrm{Res}}_{\la = \la_{nq0}} \mathscr{B}(x, \la) =
    \gamma^{\dagger}(x, \la_{nq0}) \gamma(x, \la_{nq0}) \alpha_{nq0}, \quad
    \mathop{\mathrm{Res}}_{\la = \la_{nq1}} \mathscr{B}(x, \la) = 0_m.
$$

Consider the integral
$$
    I_N(x) = \frac{1}{2 \pi i} \int\limits_{\Gamma_N} \mathscr{B}(x, \la) \, d \la,
$$
where $\Gamma_N = \{ \la \colon |\la| = (N + 1/2)^2\}$. Let us
show that for each fixed $x \in [0, \pi]$
$$\lim\limits_{N \to \iy} I_N(x) = 0_m.$$

Indeed, transforming \eqref{defgamma} similarly to \eqref{transformgamma}, and using
\eqref{defxi}, \eqref{estgamma} and the estimates
\begin{multline*}
    \| \tilde D(x, \la, \la_{klj}) \| \le \frac{C \exp(|\tau|x)}{|\rho - k| + 1}, \quad
    \| \tilde D(x, \la, \la_{klj}) - \tilde D(x, \la, \la_{kqi} \| \le \frac{C|\rho_{klj} - \rho_{kqi}|\exp(|\tau|x)}{|\rho - k| + 1}, \\
    k \ge 0, \: l, q = \overline{1, m}, \: i, j = 0, 1,
\end{multline*}
we get
$$
   \|\gamma(x, \la) \| \leq C(x) \exp(|\tau|x) \sum_{k = 0}^{\iy} \frac{\xi_k}{|\rho - k| + 1},
   \quad \mbox{Re} \, \rho \geq 0.
$$
Similarly, using \eqref{defGamma} we obtain for sufficiently large $\rho^* > 0$:
$$
    \|\Gamma(x, \la) \| \leq \frac{C(x)}{|p|} \exp(-|\tau|x) \sum_{k = 0}^{\iy} \frac{\xi_k}{|\rho - k| + 1},
    \mbox{Re} \, \rho \geq 0, \, |\rho| \geq \rho^*, \, |\rho - k| > \de > 0.
$$
Then
$$
    \| \mathscr{B}(x, \la) \| \leq \frac{C(x)} {|\rho|} \left( \sum_{k = 0}^{\iy} \frac{\xi_k}{|\rho - k| + 1} \right)^2
    \leq \frac{C(x)} {|\rho|^3}, \quad \la \in \Gamma_N.
$$
This estimate yields $\lim\limits_{N \to \iy} I_N(x) = 0_m$.

On the other hand, calculating the integral $I_N(x)$ by the residue theorem, we arrive at
$$
     \sum\limits_{k = 0}^{\iy} \sum\limits_{l = 1}^m \gamma^{\dagger}_{kl0}(x) \gamma_{kl0}(x) \alpha'_{kl0} = 0_m.
$$
Since $\alpha_{kl0} = \alpha^{\dagger}_{kl0} \geq 0$, we get
$$\gamma^{\dagger}_{kl0}(x) \gamma_{kl0}(x) \alpha_{kl0} = 0_m,$$
$$\gamma(x, \la_{kl0}) \alpha_{kl0} = 0_m, \quad k \geq 0, \quad l = \overline{1, m}.$$

Since $\gamma(x, \la)$ is entire in $\la$, and
 $$\gamma(x, \la) = O(\exp(|\tau| x))$$
 for each fixed $x \in [0, \pi]$,
 according (E), we get $\gamma(x, \la) \equiv 0_m$.
Therefore $\gamma_{nqi}(x) = 0_m$ for all $n \geq 0$, $q = \overline{1, m}$,
$i = 0, 1$, i.\,e. the homogeneous equation \eqref{homo} has only the zero solution.
\end{proof}

Thus, the proof of Theorem~\ref{thm:2} is finished.
Some discussion on condition (E) with examples is provided in \cite{Bond11}.
We can also give an alternative formulation of necessary and sufficient conditions
with condition (C) instead of (E).

\begin{thm} \label{thm:2C}
Let $\om = \om^{\dagger} \in \mathcal{D}$.
For data $\{\la_{nq}, \alpha_{nq}\}_{n \geq 0, q = \overline{1, m}} \in \mbox{Sp}$ to be
the spectral data for a certain self-adjoint problem $L \in A(\om)$ it is
necessary and sufficient to satisfy the following conditions. 

(A) The asymptotics \eqref{asymptrho}, \eqref{asymptalpha1}, 
\eqref{asymptalpha2} \eqref{asymptalpha3} are valid.

(R) The ranks of the
matrices $\al_{nq}$ coincide with the multiplicities of the corresponding
values $\la_{nq}$.

(S) All $\la_{nq}$ are real, $\alpha_{nq} = (\alpha_{nq})^\dagger$,
$\alpha_{nq} \geq 0$ for all $n \geq 0$, $q = \overline{1, m}$.
 
(C) The system of functions
$$
    \cos \rho_{nq}x \mathcal{E}_{nq}^{(i)}, \quad n \ge 0, \: q = \overline{1, m}, \: i = \overline{1, m_{nq}},
$$
is complete in $L_2((0, \pi), \mathbb{C}^m)$.

\end{thm}

Condition (C) was used by Ya.V. Mykytyuk and N.S. Trush \cite{MT10} in the characterization
of the spectral data for the self-adjoint matrix Sturm-Liouville operator
with the potential from $W_2^{-1}$.
In spite of the fact, that this class is wider than our class $L_2$,
these are two parallel results for different classes. Moreover, 
for $W_2^{-1}$ the asymptotics of eigenvalues and weight matrices are more rough,
that makes the class $W_2^{-1}$ easier for investigation.
It does not require our technic with complicated division of eigenvalues into groups. 
Now we have shown, that condition (C) can be used
in our case with our method, so there is no principal difference with the work \cite{MT10}
in this particular point.

It was established in Section~5, that (C) is weaker than (E), but (C) is equivalent to (PW).
One can easily show that
everywhere in our proofs (namely, in Lemmas~\ref{lem:complete} and \ref{lem:homo}), 
we apply this type of conditions to functions $\ga(\la) \in \mbox{PW}$.
Thus, both Theorems~\ref{thm:2} and \ref{thm:2C} are valid.
 
\medskip

{\bf Acknowledgment.} This research was supported by Grants 13-01-00134 and 14-01-31042
of Russian Foundation for Basic Research and
by the Russian Ministry of Education and Science (Grant 1.1436.2014K).

\vspace{1cm}

Natalia Bondarenko

Department of Mathematics

Saratov State University

Astrakhanskaya 83, Saratov 410026, Russia

{\it bondarenkonp@info.sgu.ru}

\end{document}